\documentclass[11pt]{article}

\usepackage{subfigure}
\usepackage{graphicx}
\usepackage{amsmath, amsthm, amssymb}
\usepackage{psfrag} 
\usepackage{array}

\usepackage[vcentering,dvips]{geometry}
\newtheorem{theorem}{Theorem}[section]

\newtheorem{corollary}[theorem]{Corollary}

\newtheorem{remark}[theorem]{Remark}

\begin{document}

\title{On Robustness Properties of Beta Encoders and Golden Ratio Encoders}

\author{Rachel Ward \thanks{
R. Ward is a graduate student in the Program in Applied and Computational Mathematics, Princeton University, Princeton, New Jersey 08544  (e-mail: rward@princeton.edu). This research project was supported by an NSF graduate research fellowship. } }

% The paper headers
%\markboth{Journal of \LaTeX\ Class Files,~Vol.~1, No.~11,~November~2002}{Shell \MakeLowercase%{\textit{et al.}}: Bare Demo of IEEEtran.cls for Journals}

\maketitle

\begin{abstract}
The beta-encoder was recently proposed as a quantization scheme for analog-to-digital conversion; in contrast to classical binary quantization, in which each analog sample $x \in [-1,1]$ is mapped to the first $N$ bits of its base-2 expansion, beta-encoders replace each sample $x$ with its expansion in a base $\beta$ between $1 < \beta < 2$.  This expansion is non-unique when $1 < \beta < 2$, and the beta-encoder exploits this redundancy to correct inevitable errors made by the quantizer component of its circuit design.  

The multiplier element of the beta-encoder will also be imprecise; effectively, the true value $\beta$ at any time can only be specified to within an interval $[ \beta_{low}, \beta_{high} ]$.  This problem was addressed by the golden ratio encoder, a close relative of the beta-encoder that does not require a precise multiplier.  However, the golden ratio encoder is susceptible to integrator leak in the delay elements of its hardware design, and this has the same effect of changing $\beta$ to an unknown value.  

In this paper, we present a method whereby exponentially precise approximations to the value of $\beta$ in both golden ratio encoders and beta encoders can be recovered amidst imprecise circuit components from the truncated beta-expansions of a "test" number $x_{test} \in [-1,1]$, and its negative counterpart, $-x_{test}$.   That is, beta-encoders and golden ratio encoders are robust with respect to unavoidable analog component imperfections that change the base $\beta$ needed for reconstruction.

\end{abstract}

\section{Introduction}
Beta-encoders with error correction were recently introduced in $\cite{3}$ as quantization algorithms in analog-to-digital (A/D) conversion that are simultaneously robust with respect to quantizer imperfections (like Sigma-Delta schemes) and efficient in their bit-budget use (like Pulse Code Modulation schemes).  For a detailed discussion comparing Sigma Delta and Pulse Code Modulation (PCM) quantization schemes, we refer the reader to $\cite{6}$ and $\cite{3}$.  Recall that after the sampling step in analog-to-digital conversion of a bandlimited function $f(t)$ with $L^{\infty}$ norm $|| f || _{\infty} \leq 1$, an N-bit PCM quantizer simply replaces each sampled function value $x = f(t_n)$ with the first N bits of its truncated binary expansion.  Beta-encoders are similar to PCM; they replace each sampled function value $x$ with a truncated series expansion in a base $\beta$, where $1 < \beta < 2$, and with binary coefficients.  Clearly, if $\beta = 2$, then this algorithm coincides with PCM.  However, whereas the binary expansion of almost every real number is unique, for every $\beta \in (1,2)$, there exist a continuum of different $\beta$ expansions of almost every $x$ in $(0,1]$ (see $\cite{8}$).  It is precisely this redundancy that gives beta-encoders the freedom to correct errors caused by imprecise quantizers shared by Sigma Delta schemes.  Whereas in Sigma Delta, a higher degree of robustness is achieved via finer sampling, beta-encoders are made more robust by choosing a smaller value of $\beta$ as the base for expansion.  
\\
\\
Although beta-encoders as discussed in $\cite{3}$ are robust with respect to quantizer imperfections, these encoders are not as robust with respect to imprecisions in other components of their circuit implementation.  Typically, beta-encoders require a multiplier in which real numbers are multiplied by $\beta$.  Like all analog circuit components, this multiplier will be imprecise; that is, although a known value $\beta_0$ may be set in the circuit implementation of the encoder, thermal fluctuations and other physical limitations will have the effect of changing the true multiplier to an unknown value $\beta \in [\beta_{low}, \beta_{high}]$ within an interval of the pre-set value $\beta_0$.   The true value $\beta$ will vary from device to device, and will also change slowly in time within a single device.  This variability, left unaccounted for, disqualifies the beta-encoder as a viable quantization method since the value of $\beta$ must be known with exponential precision in order to reconstruct a good approximation to the original signal from the recovered bit streams.
\\
\\
We overcome this potential limitation of the beta-encoder by introducing a method for recovering $\beta$ from the encoded bitstreams of a real number $x \in [-1,1]$ and its negative, $-x$.  Our method incorporates the techniques used in $\cite{2}$, but our analysis is simplified using a transversality condition, as defined in $\cite{4}$, for power series with coefficients in $\{-1,0,1\}$.  As the value of $\beta$ can fluctuate within an interval $[\beta_{low}, \beta_{high}]$ over time, our recovery technique can be repeated at regular intervals during quantization (e.g., after the quantization of every 10 samples).  
\\
\\
The golden ratio encoder (GRE) was proposed in $\cite{5}$ as a quantizer that shares the same robustness and exponential rate-distortion properties as beta-encoders, but that does not require an explicit multiplier in its circuit implementation.  GRE functions like a beta-encoder in that it produces beta-expansions of real numbers; however, in GRE, $\beta$ is fixed at the value of the golden ratio, $\beta = \phi = \frac{1+\sqrt{5}}{2}$.  The relation $\phi^2 = \phi + 1$ characterizing the golden ratio permits elimination of the multiplier from the encoding algorithm.  Even though GRE does not require a precise multiplier, component imperfections such as integrator leakage in the implementation of GRE may still cause the true value of $\beta$ to be slightly larger than $\phi$; in practice it is reasonable to assume $\beta \in [\phi, 1.1\phi]$.   Our method for recovering $\beta$ in general beta-encoders can be easily extended to recovering $\beta$ in the golden ratio encoder.
\\
\\
The paper is organized as follows:
\\
\begin{enumerate}
\item In sections I (a) and I (b), we review relevant background on beta-encoders and golden ratio encoders, respectively.
\item In section II, we introduce a more realistic model of GRE that takes into account the effects of integrator leak in the delay elements of the circuit.  We show that the output of this revised model still correspond to truncated beta-expansions of the input, but in an unknown base $\beta$ that differs from the pre-set value.
\item Section III describes a way to recover the unknown value of $\beta$ up to arbitrary precision using the bit streams of a "test" number $x \in [-1,1]$, and $-x$.  The recovery scheme reduces to finding the root of a polynomial with coefficients in $\{-1,0,1 \}$.
\item Section IV extends the recovery procedure of the previous section to the setting of beta-encoders having leakage in the (single) delay element of their implementation.  The analysis here is completely analogous to that of Section III.
\end{enumerate}

\subsection{A brief review of beta-encoders}
In this section, we summarize certain properties of beta-encoders with error correction, from the perspective of encoders which produce beta expansions with coefficients in $\{-1,1\}$.  For more details on beta-encoders, we refer the reader to $\cite{3}$.
\\
\\
We start from the observation that given $\beta \in (1,2]$, every real number $x \in [-\frac{1}{\beta - 1},\frac{1}{\beta - 1}]$ admits a sequence $(b_j)_{j \in N}$, with $b_j \in \{-1,1\}$, such that 
\begin{equation}
x = \sum_{j=1}^{\infty} b_j \beta^{-j}.
\label{ex}
\end{equation}
Under the transformation $\tilde{b}_j = \frac{b_j + 1}{2}$, $\eqref{ex}$ is equivalent to the observation that every real number $y \in [0,\frac{1}{\beta - 1}]$ admits a beta-expansion in base $\beta \in (1,2]$, with coefficients $\tilde{b}_j \in \{0,1\}$.  Accordingly, all of the results that follow have straightforward analogs in terms of $\{0,1\}$-beta-expansions; see $\cite{3}$ for more details.
\\ \\
One way to compute a sequence $(b_j)_{j \in N}$ that satisfies $\eqref{ex}$ is to run the iteration
\begin{eqnarray}
u_1 &=& \beta x \nonumber \\
b_1 &=& Q(u_1) \nonumber \\
\textrm{for $j \geq 1:$ } u_{j+1} &=& \beta(u_j - b_j) \nonumber \\
b_{j+1} &=& Q(u_{j+1}) 
\label{iter}
\end{eqnarray}
where the quantizer $Q(u)$ is simply the {\it sign}-function,
\begin{eqnarray}
	Q(u)
          &=& \left\{\begin{array}{cl}
	-1,  & u \leq 0 \\
	1,   &u > 0. 
	  \end{array}
	   \right.
          \label{(Q)}
\end{eqnarray}
For $\beta = 2$, the expansion $\eqref{ex}$ is unique for almost every $x \in [-1,1]$; however, for $ \beta \in (1,2)$, there exist uncountably many expansions of the type $\eqref{ex}$ for any $x \in [-1,1]$  (see $\cite{8}$).  Because of this redundancy in representation, beta encoders are robust with respect to quantization error, while PCM schemes are not.  We now explain in more detail what we mean by quantization error.  The quantizer $Q(u)$ in $\eqref{(Q)}$ is an idealized quantizer; in practice, one has to deal with quantizers that only approximate this ideal behavior.  A more realistic model is obtained by replacing $Q(u)$ in $\eqref{(Q)}$ with a "flaky" version $Q^{\nu}(u)$, for which we know only that
\begin{equation}
Q^{\nu}(u)  = \left\{\begin{array}{cl}
	-1,  & u  <  -\nu \\
	1,   &u  \geq \nu \\
	-1 \textrm{ or } 1, &-\nu \leq u  \leq \nu.
	   \end{array}\right.
         \label{(fqs)}
\end{equation}
In practice, $\nu$ is a quantity that is not known exactly, but over the magnitude of which we have some control, e.g. $|\nu| \leq \epsilon$ for a known $\epsilon$.  This value $\epsilon$ is called the {\it tolerance} of the quantizer.   We shall call a quantization scheme {\it robust} with respect to quantization error if, for some $\epsilon > 0$, the worst approximation error produced by the quantization scheme can be made arbitrarily small by allowing a sufficiently large bit budget, even if the quantizer used in its implementation is imprecise to within a tolerance $\epsilon$.  According to this definition, the naive $\{-1,1\}$-binary expansion is not robust.  More specifically, suppose that a flaky quantizer $Q^\nu(u)$ is used in $\eqref{iter}$ to compute the base-2 expansion of a number $x \in [-1,1]$ which is sufficiently small that $|2x| \leq \nu$.  Since $2x$ is within the "flaky" zone for $Q^{\nu}(u)$, if $b_1 = Q^{\nu}(2x)$ is assigned incorrectly; i.e., if $b_1$ differs from the sign of $x$, then no matter how the remaining bits are assigned, the difference between $x$ and the number represented by the computed bits will be at least $|x|$.   This is not the case if $1 < \beta < 2$, as shown by the following theorem whose proof can be found in $\cite{3}$:

\begin{theorem}
Let $\epsilon> 0$ and $x \in [-1,1]$.  Suppose that in the beta-encoding of $x$, the procedure $\eqref{iter}$ is followed, but the quantizer $Q^{\nu}(u)$ is used instead of the ideal $Q(u)$ at each occurence, with $\nu$ satisfying $\nu \leq  \epsilon$.  Denote by $(b_j)_{j\in N}$ the bit sequence produced by applying this encoding to the number $x$.  If $\beta$ satisfies
\begin{center}
$1 < \beta < \frac{2+\epsilon}{\epsilon+1},$
\end{center}
then
\begin{equation}
|x - \sum_{j=1}^N b_j \beta^{-j} | \leq C\beta^{-N}
\end{equation}
with $C = \epsilon + 1$.
\label{b-encod+} 
\end{theorem}
For a given tolerance $\epsilon > 0$, running the recursion $\eqref{iter}$ with a quantizer $Q^{\nu}(u)$ of tolerance $\epsilon$ and a value of $\beta$ in $(1,\frac{2+\epsilon}{\epsilon+1})$ produces bitstreams $(b_j)$ corresponding to a beta-expansion of the input $x$ in base $\beta$; however, the precise value of $\beta$ must be known in order recover $x$ from such a beta-expansion.  As mentioned in the introduction and detailed in section II, component imperfections may cause the circuit to behave as if a different value of $\beta$ is used, and this value will possibly be changing slowly over time within a known range, $[\beta_{low}, \beta_{high}]$.   In $\cite{2}$, a method is proposed whereby an exponentially precise approximation $\tilde{\gamma}$ to the value of $\gamma = \beta^{-1}$ at any given time can be encoded and transmitted to the decoder without actually physically measuring its value, via the encoded bitstreams of a real number $x \in [0,1)$ and $1-x$.  This decoding method can be repeated at regular time intervals during the quantization procedure, to account for the possible time-variance of $\gamma$.  That an exponentially precise approximation $\tilde{\gamma}$ to $\gamma$ is sufficient to reconstruct subsequent samples $f(t_n)$ with exponential precision is the content of the following theorem, which is essentially a restatement of Theorem 5 in $\cite{2}$. 

\begin{theorem}[Daubechies, Y$\i$lmaz]
Consider $x \in [0,1)$ and $(b_j)_{j \in N} \in \{0,1\}$, or $x \in [-1,1]$ and $(b_j)_{j \in N} \in \{-1,1\}$.  Suppose $\gamma \in (1/2,1)$ is such that $x = \sum_{j=1}^{\infty} b_j \gamma^j$.  Suppose $\tilde{\gamma}$ is such that $|\gamma - \tilde{\gamma}| \leq C_1 \gamma^N$ for some fixed $C_1 > 0$.    Then $\tilde{x}_N := \sum_{j=1}^N b_j \tilde{\gamma}^j$ satisfies
  \begin{equation}
  |x - \tilde{x}_N| \leq C_2 \gamma^N \end{equation} 
\label{daub}
where $C_2$ is a constant which depends only on $\gamma$ and $C_1$.
\end{theorem}

Although the approach proposed in $\cite{2}$ for estimating $\beta$ from bitstreams overcomes potential approximation error caused by imprecise multipliers in the circuit implementation of the beta-encoder, new robustness issues are nevertheless introduced. Typically, one cannot ensure that the reference level 1 in $1-x$ is known with high precision.  To circumvent this problem, the authors consider other heuristics whereby $\beta$ is recovered from clever averaging of multiple pairs $x_j$ and $1-x_j$.  These heuristics do not require that the reference level 1 in $1-x$ be precise; however, these approaches become quite complicated in and of themselves, and any sort of analytical analysis of their performance becomes quite difficult.  In section III, we will present a similar approach for recovering $\beta$ that does not require a precise reference level, yet still allows for exponentially precise approximations to $\beta$.

\subsection{The golden ratio encoder}
 As shown in the previous section, beta-encoders are robust with respect to imperfect quantizer elements, and their approximation error decays exponentially with the number of bits $N$.   To attain this exponential precision, $\beta$ must be measured with high precision, which is quite complicated in practice.   These complications motivated the invention of the golden ratio encoder (GRE) of $\cite{5}$, which has the same robustness and rate-distortion properties as beta-encoders, but uses an additional delay element in place of precise multiplier in its implementation.  More precisely, if one implements the recursion $\eqref{iter}$ with $\beta = \phi = \frac{1 + \sqrt{5}}{2}$, then using the relation $\phi^2 = \phi + 1$, one obtains the recursion formula $u_{n+1} = u_{n-1} + u_n - (b_{n-1} + \phi b_n)$.  If the term $b_{n-1} + \phi b_n$ in this formula is removed, then the resulting recursion $v_{n+1} = v_n + v_{n-1}$ should look familiar; indeed, with initial conditions $(v_0,v_1) = (0,1)$, this recursion generates the Fibonacci numbers $v_n$, and it is well-known that the sequence $\frac{v_{n+1}}{v_n} \rightarrow \phi$ as $n \rightarrow \infty$.   If $b_{n-1} + \phi b_n$ is instead replaced by a single bit taking values in $\{-1,1\}$, then we are led to the following scheme:
\begin{eqnarray}
            u_0 &=& 0    \nonumber \\
            u_1 &=& x   \nonumber \\
            b_1 &=& Q(0,x)      \nonumber \\	
\textrm{for $n \geq 1:$ } u_{n+1} &=& u_{n} + u_{n-1} - b_{n} \nonumber \\
	  b_{n+1} &=& Q(u_n, u_{n+1})  
	  \label{(simple recurs)} 
        \end{eqnarray} 
In this paper, we will consider quantizers $Q(u,v)$ in $\eqref{(simple recurs)}$ of the form $Q_{\alpha}(u,v)$, where
\begin{equation}
Q_{\alpha}(u,v)  = \left\{\begin{array}{cl}
	-1,  & u + \alpha v  <  0\\
	1,   &u + \alpha v \geq 0
	   \end{array}\right.
           \label{(fq2)}
\end{equation}
along with their flaky analogs,
\begin{equation}
Q_{\alpha}^{\nu}(u,v)  = \left\{\begin{array}{cl}
	-1,  & u + \alpha v  <  -\nu \\
	1,   &u + \alpha v \geq \nu \\
	-1 \textrm{ or } 1, &-\nu \leq u + \alpha v \leq \nu
	   \end{array}\right.
           \label{(fq2)}
\end{equation}
In $\cite{5}$, the authors consider the recursion formula $\eqref{(simple recurs)}$ implemented with flaky $\{0,1\}$-quantizers  of the form $\bar{Q}_{\alpha}^{\nu,\iota}(u,v) = \Big[ Q^{\nu}(u + \alpha v - \iota) + 1 \Big]/2$.  For the simplicity of presentation, we will consider only the $\{-1,1\}$-quantizers $\eqref{(fq2)}$, but many of our results extend straightforwardly to quantizers of the type $\bar{Q}_{\alpha}^{\nu,\iota}(u,v)$. 
\\
\\
The following theorem was proved in $\cite{5}$; it shows that as long as $x$ and $Q(u,v)$ are such that the state sequence $u = \{u_n\}_{n=0}^{\infty}$ remains bounded, a golden ratio encoder (corresponding to the recursion $\eqref{(simple recurs)}$) will produce a bitstream $(b_j)$ corresponding to a beta-expansion of $x$ in base $\beta = \phi$, just as does the beta-encoder from which the GRE was derived.  

\begin{theorem}[Daubechies, G{\"u}nt{\"u}rk, Wang, Y$\i$lmaz]
Consider the recursion $\eqref{(simple recurs)}$.  Suppose the 1-bit quantizer $Q(u,v)$ that outputs bits $(b_j)$ in $\{-1,1\}$ is of the type $Q_{\alpha}^{\nu}(u,v)$ such that the state sequence $u = \{u_n\}_{n=0}^{\infty}$ with $u_0 = 0$ and $u_1 = x$ is bounded.   Then
\begin{equation}
 |x - \sum_{n=1}^{N} b_n {\phi}^{-n} | \leq \phi^{-N+1}
\label{expand}
\end{equation}
Here $\phi$ is the golden ratio.
\end{theorem}

\begin{proof} 
Note that
\begin{eqnarray}
\sum_{n=1}^{N} b_{n} \phi^{-n} &=& \sum_{n=1}^{N} (u_{n-1} + u_{n} - u_{n+1}) \phi^{-n} \nonumber \\
&=& \sum_{n=0}^{N-1} u_n \phi ^{-(n+1)} +  \sum_{n=1}^{N} u_n \phi^{-n}  -\sum_{n=2}^{N+1} u_n \phi ^{-(n-1)} \nonumber \\
\nonumber \\
&=& u_0 \phi^{-1} + (\phi^{-1} + 1  - \phi) \sum_{n=2}^{N-1} \phi^{-n} u_n  \nonumber \\  
&&+ u_1(\phi^{-1}  +  \phi^{-2}) + \phi^{-N} \Big(u_{N} (1 - \phi) - u_{N+1} \Big) \nonumber \\
&=& u_0 \phi^{-1} + u_1 + \phi^{-N} \Big(u_{N} (1 - \phi) - u_{N+1} \Big) \nonumber \\
&=& x  +\phi^{-N}\Big( u_{N} (1 - \phi) - u_{N+1} \Big).  \nonumber
\end{eqnarray}
The second to last equality uses the relation $\phi^{-2} + \phi^{-1} - 1 = 0$, and the last equality is obtained by setting $u_0 = 0$ and $u_1 = x$.  Since the state sequence $u = \{u_n\}_{n=0}^{\infty}$ is bounded, it follows that $x =  \sum_{n=1}^{\infty} b_n \phi^{-n}$ .   Thus,
\begin{eqnarray}
 |x - \sum_{n=1}^{N} b_n {\phi}^{-n} | &=&  |\sum_{n=N+1}^{\infty} b_n \phi^{-n} | \nonumber \\
 &\leq& \frac{\phi^{-(N+1)}}{1-\phi^{-1}} \nonumber \\
 &=& \phi^{-N+1}.
 \end{eqnarray}
\begin{flushright}
$\blacksquare$
 \end{flushright}
 \end{proof} 
Figure 1 compares the block-diagrams of a beta-encoder with error correction, corresponding to the recursion formula $\eqref{iter}$, and a golden ratio encoder corresponding to the recursion $\eqref{(simple recurs)}$.  Although the implementation of GRE requires 2 more adders and one more delay element than the implementation of the beta-encoder, the multiplier element $\alpha$ in GRE does not have to be precise (see section 6), whereas imprecisions in the multiplier element $\beta$ of the beta-encoder result in beta-expansions of the input $x$ in a different base $\beta'$.  

\begin{figure*}
\centerline{\subfigure[Beta-encoder]{\includegraphics[width=2.5in]{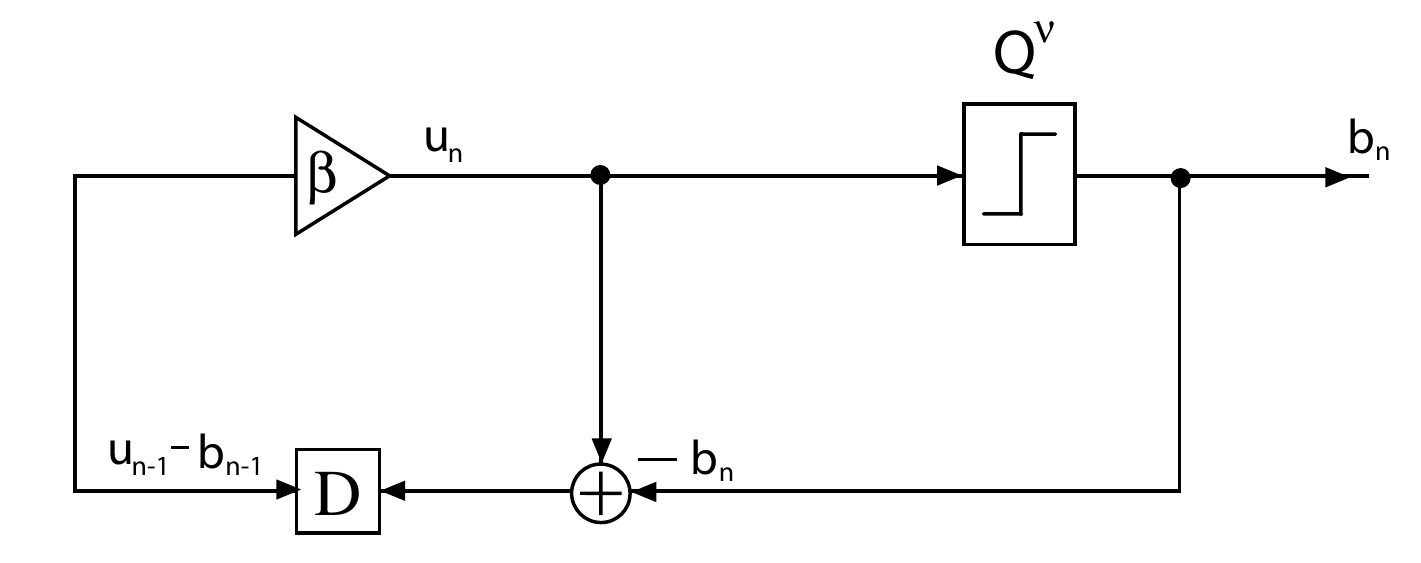}
\hfil
\label{one}}
\subfigure[GRE-encoder]{\includegraphics[width=2.5in]{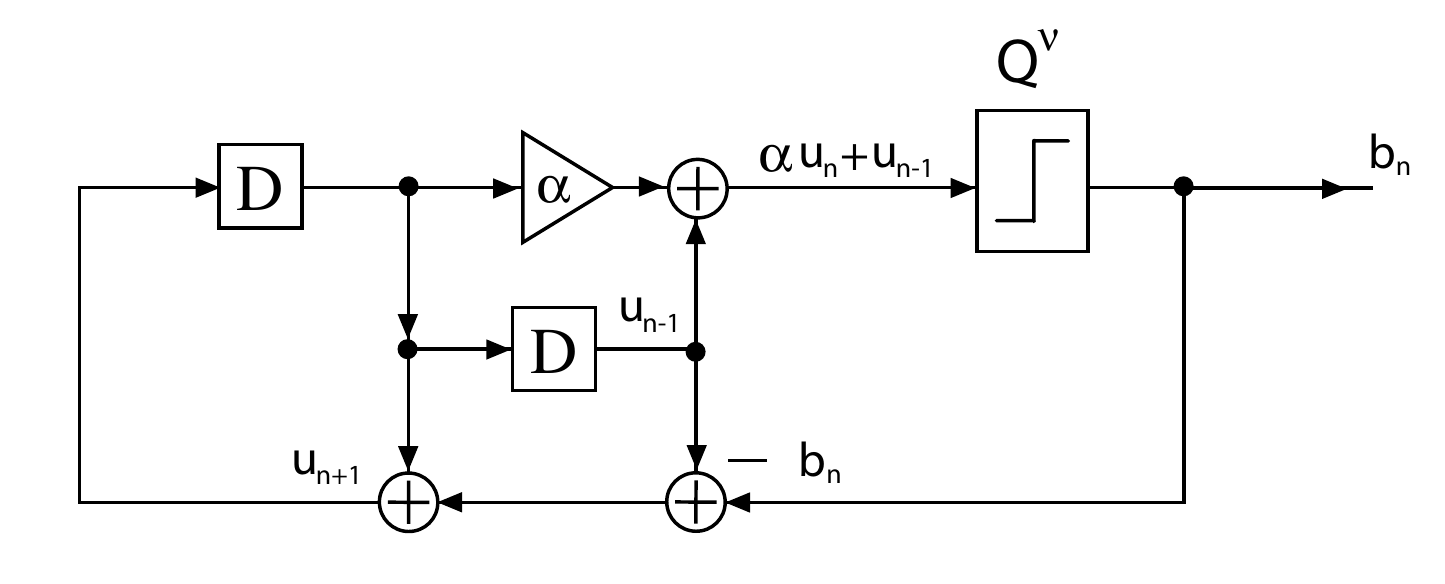}
\label{2}}}
\caption{ {\small A comparison of the block diagrams for (a) a beta-encoder with error correction, and a (b) golden ratio encoder.  Here $D$ represents a delay term, which holds input throughout one timestep. }}
\label{figure 1}
\end{figure*}
  
\section{GRE: The Revised Scheme}
In modeling the golden ratio encoder by the system $\eqref{(simple recurs)}$, we are assuming that the delay elements used in its circuit implementation are ideal.  A more realistic model would take into account the effect of integrator leak, which is inevitable in any practical circuit implementation (see $\cite{9}$ for more details).  After one clock time, the stored input in the first delay is reduced to $\lambda_1$ times its original value, while the stored input in the second delay is replaced by $\lambda_2$ times its original value.  In virtually all circuits of interest, no more than $10$ percent of the stored input is leaked at each timestep; that is,  we can safely assume that $\lambda_1$ and $\lambda_2$ are parameters in the interval [.9,1].  The precise values of these parameters may change in time; however, as virtually all practical A/D converters produce over 1000 bits per second (and some can produce over 1 billion bits per second),  we may safely assume that $\lambda_1$ and $\lambda_2$ are constant throughout the quantization of at least every 10 samples.  
\\
\\
Fixing an input value $x \in [-1,1]$, we arrive at the following revised description of the GRE encoder: 
\begin{eqnarray}
           u_0 &=& 0    \nonumber \\
           u_1 &=& x   \nonumber \\
           b_1 &=& Q(0,x)      \nonumber \\		
        \textrm{for $n \geq 1:$ }   u_{n+1} &=& \lambda_1 u_{n} +\lambda_1 \lambda_2 u_{n-1} - b_{n} \nonumber \\
	  b_{n+1} &=& Q(u_{n}, u_{n-1}) 
	  \label{(lambda recurs)}
\end{eqnarray}	
Obviously, $\lambda_1 = \lambda_2 = 1$ corresponds to the original model $\eqref{(simple recurs)}$.  It is reasonable to assume in practice that $(\lambda_1,\lambda_2) \in M := [.95,1]^2$, and in virtually all cases $(\lambda_1,\lambda_2) \in V := [.9,1]^2$.
\\
\\
We will show that the revised scheme $\eqref{(lambda recurs)}$ still produces beta-expansions of the input $x$, but in a slightly different base $\gamma = \beta^{-1} = \frac{-\lambda_1 + \sqrt{\lambda_1^2 + 4 \lambda_1 \lambda_2}}{2 \lambda_1 \lambda_2}$, which increases away from $\phi^{-1}$ as the parameters $\lambda_1$ and $\lambda_2$ decrease.   Key in the proof of Theorem $\eqref{expand}$ was the use of the relation $1 - \phi^{-1} - \phi^{-2} = 0$ to reduce $\sum_{n=1}^N (u_n + u_{n+1} - u_{n+2}) \phi^{-n}$ to the sum of the input $x$, and a remainder term that becomes arbitrarily small with increasing $N$.   Accordingly, the relation $1  - \lambda_1 \gamma - \lambda_1 \lambda_2 \gamma^2 = 0$ gives $\sum_{n=1}^N (u_n + u_{n+1} - u_{n+2}) \gamma^n = x + R(N)$, where $R(N)$ goes to $0$ as $N$ goes to infinity.   
\begin{theorem}    Suppose the 1-bit quantizer $Q$ in $\eqref{(lambda recurs)}$ of type $\eqref{(fq2)}$ is such that the state sequence $u = \{u_n\}_{n=0}^{\infty}$ with $u_0 = 0$ and $u_1 = x$ is bounded.  Consider $\gamma = \frac{-\lambda_1 + \sqrt{\lambda_1^2 + 4 \lambda_1 \lambda_2}}{2 \lambda_1 \lambda_2}$. Then
\begin{center}
$|x - \sum_{n=1}^{N} b_n \gamma^n | \leq C_{\gamma} \gamma^{N} $
 \end{center}
 where $C_{\gamma} = \frac{\gamma}{1 - \gamma}$.
\label{exp}
\end{theorem}	

\begin{proof}
\begin{eqnarray}
\sum_{n=1}^{N} b_{n} \gamma^{n} &=& \sum_{n=1}^{N} ( \lambda_1 \lambda_2 u_{n-1} + \lambda_1 u_{n} - u_{n+1}) \gamma ^{n} \nonumber \\
&=&  (\lambda_1 \lambda_2 \gamma + \lambda_1  - \gamma^{-1}) \sum_{n=2}^{N-1} \gamma^n u_n \nonumber \\
&  & + \lambda_1 \lambda_2 u_0 \gamma + u_1(\lambda_1 \gamma  + \lambda_1 \lambda_2 \gamma^2) \nonumber \\
&  & + \gamma^{N} \Big(u_{N} (\lambda_1 - \gamma^{-1}) - u_{N+1} \Big) \nonumber \\
 \nonumber \\
&=&  x  +\gamma^{N}\Big( u_{N} (\lambda_1 - \gamma^{-1}) - u_{N+1} \Big).
\end{eqnarray}
The last equality is obtained by setting $u_0 = 0$ and $u_1 = x$. 
Since the $u_n$ are bounded, it follows as in the proof of $\eqref{(simple recurs)}$ that 
\begin{eqnarray}
 |x - \sum_{n=1}^{N} b_n \gamma^n | &\leq& \frac{\gamma^{N+1}}{1 - \gamma}.
\end{eqnarray}
\begin{flushright}
$\blacksquare$
 \end{flushright} \end{proof}

Theorem $\eqref{exp}$ implies that if $\gamma$ is known, then the revised GRE scheme $\eqref{(lambda recurs)}$ still gives exponential approximations to the input signal $x$, provided that the $u_n$ are indeed bounded.   The following theorem gives an explicit range for the parameters $(\nu, \alpha)$ which results in bounded sequences $u_n$ when the input $x \in [-1,1]$, independent of the values of the leakage parameters $(\lambda_1,\lambda_2)$ in the set $V = [.9,1]^2$.  This parameter range is only slightly more restrictive than that derived in $\cite{5}$ for the ideal GRE scheme $\eqref{(simple recurs)}$; that is, the admissable parameter range for $\alpha$ and $\nu$ is essentially robust with respect to leakage in the delay elements of the GRE circuit implementation.  
\begin{theorem}
Let $x \in [-1,1]$, and $(\lambda_1, \lambda_2) \in [.9, 1]^2$.  Suppose that the GRE scheme $\eqref{(lambda recurs)}$ is followed, and the quantizer $Q_{\alpha}^{\nu}(u,v)$ is used, with $\nu$ possibly varying at each occurrence, but always satisfying $\nu \leq \epsilon$ for some fixed tolerance $\epsilon \leq .5$.    If the parameter $\alpha$ takes values in the interval $[1.198+1.198 \epsilon, 2.281 - .952 \epsilon]$, then the resulting state sequence $(u_j)_{j \in N}$ is bounded.
\label{range}
\end{theorem}
We leave the proof of Theorem $\eqref{range}$ to the appendix.  This theorem in some sense parallels  Theorem $\eqref{b-encod+}$ in that an admissable range for the "multiplication" parameter  ($\alpha$ in GRE, $\beta$ in beta-encoders) is specified for a given quantizer tolerance $\epsilon$; however, we stress that the specific value of $\beta$ is needed in order to recover the input from the bitstream $(b_j)$ in beta-encoders.  In contrast, a GRE encoder can be built with a multiplier $\alpha$ set at any value within the range $[1.198+1.198 \epsilon, 2.281 - .952 \epsilon]$, and as long as this multiplier element  has enough precision that the true value of $\alpha$ will not stray from an interval $[\alpha_{low}, \alpha_{high}] \subset [1.198+1.198 \epsilon, 2.281 - .952 \epsilon]$, then the resulting bitstreams $(b_j)$ will always represent a series expansion of the input $x$ in base $\gamma$ of Theorem $\eqref{exp}$, which does not depend on $\alpha$.  The base $\gamma$ does however depend on the leakage parameters $\lambda_1$ and $\lambda_2$, which are not known a priori to the decoder and can also vary in time from input to input: as discussed earlier, the only information available to the decoder a priori is that $(\lambda_1,\lambda_2) \in [.9,1]^2$ in virtually all cases of interest, and $(\lambda_1,\lambda_2) \in [.95,1]^2$, or $\gamma \in [\phi^{-1}, \frac{20}{19} \phi^{-1}] \approx [.618, .6505]$ in most cases of interest.  In the next section, we show that the upper bound of $.6505$ is sufficiently small that the value of $\gamma$ can be recovered with exponential precision from the encoded bitstreams of a pair of real numbers $(x,-x)$.  In this sense, GRE encoders are robust with respect to leakage in the delay elements, imprecisions in the multiplier $\alpha$, { \it and } quantization error.

\section{Determining $\gamma$}
\subsection{Approximating $\gamma$ using an encoded bitstream for x = 0}
Recall that by Theorem $\eqref{daub}$, exponentially precise approximations $\tilde{\gamma}$ to the root $\gamma$ in Theorem $\eqref{exp}$ are sufficient in order to reconstruct subsequent input $x_n = f(t_n) \in [-1,1]$ whose bit streams are expansions in root $\gamma$ with exponential precision.  In this section, we present a method to approximate $\gamma$ with such precision using the only information at our disposal at the decoding end of the quantization scheme: the encoded bitstreams of real numbers $x \in [-1,1]$.  More precisely, we will be able to recover the value $\gamma = \beta^{-1} $ using only a single bitstream corresponding to a beta-expansion of the number 0.  It is easy to adapt this method to slow variations of $\gamma$ in time, as one can repeat the following procedure at regular time intervals during quantization, and update the value of $\gamma$ accordingly.  
\\
\\
 The analysis that follows will rely on the following theorem by Peres and Solomyak $\cite{1}$:
\begin{theorem}[Peres-Solomyak] [$\delta$-transversality]
Consider the intervals $I_{\rho} = [0,\rho]$.  If $\rho \leq .6491...$, then for any $g$ of the form
\begin{equation}
g(x) = 1 + \sum_{j=1}^{\infty} b_j x^j \textrm{,                    }  b_j \in \{-1,0,1\}
\label{g}
\end{equation}
and any $x \in I_{\rho}$, there exists a $\delta_{\rho} > 0$ such that if $g(x) < \delta_{\rho}$ then $g'(x) < -\delta_{\rho}$.  Furthermore, as $\rho$ increases to $.6491...$, $\delta_{\rho}$ decreases to 0.
\label{trans}
\end{theorem}

Theorem $\eqref{trans}$ has the following straightforward corollary:
\begin{corollary}
If $g(x)$ is a polynomial or power series belonging to the class $B$ given by
\begin{equation}
B = \{ \pm 1 + \sum_{j=1}^{\infty} b_j x^j  \textrm{,     }: b_j \in \{-1,0,1\} \}
\label{g+}
\end{equation}
then $g$ can have no more than one root on the interval $(0, .6491]$.  Furthermore, if such a root exists, then this root must be simple.
\label{root}
\end{corollary}
In $\cite{1}$, Peres and Solomyak used Theorem $\eqref{trans}$ to show that the distribution $\nu_\lambda$ of the random series  $\sum \pm \lambda^n$ is absolutely continuous for a.e. $\lambda \in (1/2, 1)$.   The estimates in Theorem $\eqref{trans}$ are obtained by computing the smallest double zero of a larger class of power series $\bar{B} := \{1 + \sum_{n=1}^{\infty} b_nx^n : b_n \in [-1,1] \}$.   The specific upper bound $\rho = .6491...$ for which $\delta$-transversality holds on the interval $I_{\rho}$ is the tightest bound that can be reached using their method of proof, but the true upper bound cannot be much larger; in $\cite{7}$, a power series $g(t)$ belonging to the class $B$ in $\eqref{g+}$ is constructed which has a double zero at $t_0 \approx .68$ (i.e., $g(t_0) = 0$ and $g'(t_0) = 0$), and having a double root obviously contradicts $\delta$-transversality.
\\
\\
We now show how to use Theorem $\eqref{trans}$ to recover $\gamma$ from a bitstream $(b_j)_{j=1}^{\infty}$ produced from the GRE scheme $\eqref{(lambda recurs)}$ corresponding to input $x=0$.  We assume that the parameters $(\alpha,\nu)$ of the quantizer $Q_{\alpha}^{\nu}(u,v)$ used in this implementation are within the range provided by Theorem $\eqref{range}$.  Note that such a bitstream $(b_j)_{j=1}^{\infty}$ is not unique if $\nu > 0$, or if more than one pair $(\lambda_1,\lambda_2)$ correspond to the same value of $\gamma$.   Nevertheless, Theorem $\eqref{exp}$ implies that $ \sum_{j=1}^{\infty} b_j \gamma^{j} = 0$, so that $\gamma$ is a root of the power series $F(t) = b_1 + \sum_{j=1}^{\infty} b_{j+1} t^j$.  Suppose we know that $\beta \geq 1.54056$, or that $\gamma \leq .6491$.  Since $F(t)$ belongs to the class $B$ of Corollary $\eqref{root}$, $\gamma$ must necessarily be the { \it smallest} positive root of $F(t)$.  In reality one does not have access to the entire bitsream $(b_j)_{j=1}^{\infty}$, but only a finite sequence $(b_j)_{j=1}^{N}$.  It is natural to ask whether we can approximate the first root of $F(t)$ by the first root of the polynomials $P_n(t) = b_1 + \sum_{j=1}^{n} b_{j+1} t^j$.   Since the $P_n$ still belong to the class $B$ of Corollary $\eqref{root}$, $|P_n(t)|$ has {\it at most} one zero on the interval $[\phi^{-1},.6491] \approx [.618, .6491]$.  The following theorem shows that, if it is known a priori that $\gamma \leq .6491 - \epsilon$ for some $\epsilon > 0$, then for $n$ sufficiently large, $|P_n(t)|$ is guaranteed to have a root $\gamma_n$ in $[0, .6491]$, and  $| \gamma - \gamma_n |$ decreases exponentially as $n$ increases.  

\begin{theorem}
Suppose that for some $\epsilon > 0$,  it is known that $\gamma \leq \gamma_{high} = .6491 - \epsilon$.   Let $\delta > 0$ be such that $\delta$-transversality holds on the interval $[0, \gamma_{high}]$.  Let $N$ be the smallest integer such that $\gamma^{N+1} \leq (1 - \gamma) \epsilon \delta$.  Then for $n \geq N$, \begin{enumerate}
\renewcommand{\labelenumi}{(\alph{enumi})}
\item $P_n$ has a unique root $\gamma_n$ in $[0,.6491]$
\item $|\gamma - \gamma_n| \leq C_1 \gamma^n$, where $C_1 = \frac{\gamma}{\delta(1 - \gamma)}$.
\end{enumerate}  
\label{mytheorem}
\end{theorem}

\begin{proof}
Without loss of generality, we assume $b_1 = 1$, and divide the proof into 2 cases: (1) $P_n(\gamma) \leq 0$, and (2) $P_n(\gamma)  > 0$.   The proof is the same for $b_1 = -1$, except that the cases are reversed.
\\
\\
{\bf Case (1)}: In this case, many of the restrictions in the theorem are not necessary; in fact, the theorem holds here for all $n > 0$ and $\epsilon \geq 0$.    $P_n(t)$ has opposite signs at $t=0$ and $t=\gamma$, so $P_n$ must have at least one root $\gamma_n$ in between.  Moreover, this root is unique by Theorem $\eqref{trans}$.   To prove part (b) of the theorem, observe that  Theorem $\eqref{trans}$ implies that if $P_n(t_0) \leq \delta$ at some $t_0 \in (0,.6491)$, then $P_n'(t) \leq - \delta$ in the interval $[t_0,.6491)$.  In particular, $P_n(t) \leq \delta$ and $P_n'(t) \leq - \delta$ for $t \in [\gamma_n, \gamma]$.  By the Mean Value Theorem, $|P_n(\gamma)| = |P_n(\gamma) - P_n(\gamma_n)| = |P_n'(\xi)| |\gamma - \gamma_n|$ for some $\xi \in [\gamma_n, \gamma]$, so that
\begin{eqnarray}
    |\gamma - \gamma_n| &\leq&  \frac{|P_n(\gamma)|}{\inf_{\xi \in [\gamma_n,\gamma]} |P_n'(\xi)|} \nonumber \\
    &\leq&  \frac{|P_n(\gamma)|}{\delta} \nonumber \\ 
    &\leq&  \frac{\gamma^{n+1}}{(1 - \gamma)\delta} .
    \nonumber
  \end{eqnarray}  
 The inequality $|P_n(\gamma)| \leq \frac{\gamma^{n+1}}{1-\gamma}$ follows from Theorem $\eqref{exp}$.
\\
\\
{\bf Case (2)}:  If $P_n(\gamma) > 0$, then Theorem $\eqref{trans}$ implies that the first positive root of $P_n$, if it exists, must be greater than $\gamma$.   Whereas in Case (1), $P_n$ was guaranteed to have a root $\gamma_n \leq \gamma$ for all $n \geq 0$, in this case $P_n$ might not even have a root in $(0,1)$; for example, the first positive root of the polynomial $P_4(t) = 1 - t - t^2 + t^3$ occurs at $t = 1$.   However, if $n$ is sufficiently large, $P_n$ will have a root in $[\gamma, .6491]$.  Precisely, let $n \geq N$, where $N$ is the smallest integer such that $\gamma^{N+1} \leq (1 - \gamma) \epsilon \delta$.  Then $P_n(\gamma) \leq \frac{\gamma^{n+1}}{1-\gamma} \leq \frac{\gamma^{N+1}}{1-\gamma}  \leq \epsilon \delta \leq \delta$.  Since $\gamma + \epsilon \leq .6491$, Theorem $\eqref{trans}$ gives that $P_n'(t) \leq - \delta$ for $t \in [\gamma, \gamma + \epsilon]$, and    
\begin{eqnarray}
P_n(\gamma + \epsilon) &\leq& P_n(\gamma) + \epsilon \sup_{t \in [\gamma, \gamma+\epsilon] } P_n'(t) \nonumber \\ &\leq& \epsilon \delta - \epsilon \delta \nonumber \\  &=& 0 \nonumber.
\end{eqnarray}
We have shown that $P_n(\gamma) > 0$ and $P_n(\gamma + \epsilon) \leq 0$, so $P_n$ is guaranteed to have a root $\gamma_n$ in the interval $[\gamma, \gamma + \epsilon]$, which is in fact the unique root of $P_n$ in $[0,.6491]$.  Furthermore, the discrepency between $\gamma_n$ and $\gamma$ becomes exponentially small as $n > N$ increases, as
\begin{center}
$| \gamma_n - \gamma | \leq \frac{|P_n(\gamma)|}{\inf_{\xi \in [\gamma,\gamma_n]} |P_n'(\xi)|} 
  \leq  \frac{\gamma^{n+1}}{(1 - \gamma)\delta}$  .
  \end{center}
\begin{flushright}
$\blacksquare$
 \end{flushright}
 \end{proof}
\begin{remark} In $\cite{1}$ it is shown that $\delta$-transversality holds on $[0, .63]$ with $\delta = \delta_{.63} = .07$.  This interval corresponds to $\epsilon = .6491 - .63 = .0191$ in Theorem $\eqref{mytheorem}$.   If $\gamma$ is known a priori to be less than $.63$, then $\frac{\gamma}{1 - \gamma} \leq \frac{.63}{1-.63} \approx 1.7027$, and $N$ of Theorem $\eqref{mytheorem}$ satisfies
\begin{center}
 $N \leq \frac{\log{\epsilon\delta} - \log{1.7027}}{\log{.63}} \leq 16$.
 \end{center}
\end{remark}
 
 Of course, even if we know $P_n(t)$,  $n$ will be too large to solve for $\gamma_n$ analytically.  However, if we can approximate $\gamma_n$ by $\tilde{\gamma}_n$ with a precision of $O(\gamma^n)$, then $|\gamma - \tilde{\gamma}_n| \leq |\gamma - \gamma_n| + |\gamma_n - \tilde{\gamma}_n|$ will also be of order $\gamma^n$, so that the estimates $\tilde{\gamma}_n$ are still exponentially accurate approximations to the input signal.   
\\
\\
Indeed, any $\tilde{\gamma}_n \in [\phi^{-1},.6491 - \epsilon]$ satisfying $| P_n(\tilde{\gamma}_n) | \leq \phi^{-n}$ provides such an exponential approximation to $\gamma_n$.  Since $n \geq N$, it follows that $\phi^{-n} \leq \gamma^n \leq \gamma^N \leq \frac{(1 - \gamma) \epsilon \delta}{\gamma} \leq \delta$, and so $P_n'(t) \leq -\delta$ on the interval between $\tilde{\gamma}_n$ and $\gamma_n$.  Thus, 
\begin{eqnarray}
 | \gamma - \tilde{\gamma}_n | &\leq& | \gamma - \gamma_n | + | \gamma_n - \tilde{\gamma}_n | 
 \nonumber \\
 &\leq& C_1 \gamma^{n} +  \frac{|P_n(\tilde{\gamma}_n)|}{\inf_{\xi \in [\tilde{\gamma}_n,\gamma_n] or \xi \in [\gamma_n, \tilde{\gamma}_n]} |P_n'(\xi)|}  \nonumber \\
 &\leq& C_1 \gamma^{n} + \frac{\phi^{-n}}{\delta} \nonumber \\
 & \leq& C_1 \gamma^{n} + \frac{\gamma^n}{\delta} \nonumber \\
 & =  & C_2 {\gamma^n}
 \label{search}
\end{eqnarray}
where $C_2 = C_1 + \frac{1}{\delta} = \frac{1}{\delta(1 - \gamma)}$.

\subsection{Approximating $\gamma$ with beta-expansions of $(-x,x)$}

The method for approximating the value of $\gamma$ in the previous section requires a bitstream $b$ corresponding to running the GRE recursion $\eqref{(lambda recurs)}$ with specific input $x = 0$.  This assumes that the reference value $0$ can be measured with high precision, which is an impractical constraint.  We can try to adapt the argument using bitstreams of an encoded pair $(x,-x)$ as follows.   Let   $b$ and $c$ be bitstreams corresponding to $x$ and $-x$, respectively.  Define $d_j = b_j + c_j$. Put $k = min \{j \in N |  d_{j+1} \neq 0 \}$, and consider the sequence $\bar{d} = (\bar{d}_j)_{j=1}^{\infty}$ defined by $\bar{d}_j = \frac{1}{2} d_{j+k}$.  Since $d_j \in \{-2,0,2\}$, it follows that $\bar{d}_j \in \{-1,0,1\}$, and by Theorem $\eqref{exp}$, we have that
\begin{eqnarray}
\sum_{n=1}^{\infty} \bar{d}_n \gamma^n &=&  \frac{1}{2} \gamma^{-k} \sum_{n=1}^{\infty} d_n \gamma^n \nonumber \\
          &=& \frac{1}{2} \gamma^{-k} \Big[ \sum_{n=1}^{\infty} b_n \gamma^n + \sum_{n=1}^{\infty} c_n \gamma^n \Big]  \nonumber \\
             &=& 0      
             \end{eqnarray}                  
so that 
\begin{equation}
| \sum_{n=1}^{N} \bar{d}_n \gamma^n | \leq C_{\gamma} \gamma^N.
\label{close}
\end{equation}
where the constant $C_{\gamma} = \frac{\gamma}{1 - \gamma}$.

Equation $\eqref{close}$, along with the fact that the polynomials $\bar{P}_N(t) :=  \bar{d}_1 + \sum_{j=1}^{N} \bar{d}_{j+1} t^j $ are of the form $\eqref{g+}$, allows us to apply Theorem $\eqref{mytheorem}$ to conclude that for $N$ sufficiently large, the first positive root $\gamma_N$ of $\bar{P}_N$ becomes exponentially close to $\gamma$.  However, note that the encoding of $(x,-x)$ is {\it not} equivalent to the encoding of $0$.  The value of $k = min \{j \in N |  d_{j+1} \neq 0 \}$ used to define the sequence $\bar{d}$ can be { \it arbitrarily large}.  In fact, if an ideal quantizer $Q_{\alpha}^0$ is used, then the bitstreams $b$ and $c$ are uniquely defined by $b \equiv -c$, so that $d \equiv 0$.  Thus, this method for recovering $\gamma$ actually { \it requires} the use of a flaky quantizer $Q_{\alpha}^{\nu}$.  To this end, one could intentially implement GRE with a quantizer which toggles close to, but not exactly at zero.  One could alternatively send not only a single pair of bitstreams $(b,c)$, but {\it multiple} pairs of bitstreams $(b^l, c^l)$ corresponding to several pairs $(x_l,-x_l)$, to increase the chance of having a pair that has $b^l,c^l \neq 0$ for relatively small $j$. 
\\
\\
Figure 3 plots several instances of $\bar{P}_{8}, \bar{P}_{16}$, and $\bar{P}_{32}$, corresponding to $\gamma = .64375$.  The quantizer $Q_{\alpha}^{\nu}(u,v)$ is used, with $\alpha = 2$ and $\nu = .3$.  These values of $\alpha$ and $\nu$ generate bounded sequences $(u_n)$ for all $(\lambda_1,\lambda_2) \in [.9,1]^2$ by Theorem $\eqref{range}$.
\\
\\
As shown in Figure 2, numerical evidence suggests that 10 iterations of Newton's method starting from $x_0 = \phi^{-1} \approx .618$ will compute an approximation to $\gamma_N$, the first root of $\bar{P}_N$, with the desired exponential precision.  The figure plots $\gamma$ versus the error $| \gamma - \tilde{\gamma}_N|$, where $\tilde{\gamma}_N$ is the approximation to $\gamma_N$ obtained via a 10- step Newton Method, starting from $x_0 = .618$.  More precisely, for each $N$, we ran $100$ different trials, with $x$ and $\gamma$ picked randomly from the intervals $[-1,1]$ and $[.618,.7]$ respectively, and independently for each trial.   The worst case approximation error of the 100 trials is plotted in each case.  The quantizer used is $Q_{\alpha}^{\nu}(u,v)$ with $\nu = .3$, and $\alpha$ picked randomly in the interval $[1.7,2]$, independently for each trial.  Again, Theorem $\eqref{range}$ shows that these values of $\alpha$ and $\nu$ generate bounded sequences $(u_n)$ for all $(\lambda_1,\lambda_2) \in [.9,1]^2$.

\section{Beta-expansions Revisited}
Even though our analysis of the previous section was motivated by leaky GRE encoders, it can be applied to general beta-encoders to recover the value of $\beta$ at any time during quantization.  From the last section, we have: 
\begin{theorem}
Let $F(t) = \sum_{j=0}^{\infty}b_j t^j$ be a power series belonging to the class $B = \{ \pm 1 + \sum_{j=1}^{\infty} b_j x^j  \textrm{,     } b_j \in \{-1,0,1\} \}$.  Suppose that $F$ has a root at $\gamma \in [\gamma_{low},\gamma_{high}]$, where $\gamma_{high} = .6491 - \epsilon$ for some $\epsilon > 0$.   Let $\delta > 0$ be such that $\delta$-transversality holds on the interval $[0, \gamma_{high}]$.  Let $N$ be the smallest integer such that $\gamma^{N+1} \leq \epsilon \delta (1 - \gamma)$. Then for $n \geq N$, \begin{enumerate}
\renewcommand{\labelenumi}{(\alph{enumi})}
\item The polynomials $P_n(t) = \sum_{j=0}^{n}b_j t^j$ have a unique root $\gamma_n$ in $[0,.6491]$
\item Any $\tilde{\gamma} \in [\gamma_{low}, \gamma_{high}]$ which satisfies $| P_n(\tilde{\gamma}) | \leq (\gamma_{low})^n$ also satisfies $|\tilde{\gamma} - \gamma| \leq \tilde{C}|\gamma|^n$, where $\tilde{C} = \frac{1}{\delta(1 - \gamma)}$.
\end{enumerate}  
\label{mybigtheorem}
\end{theorem}
This theorem applies to beta encoders, corresponding to implementing the recursion $\eqref{iter}$ with flaky quantizer $Q^{\nu}(u)$ defined by $\eqref{(fqs)}$, and with $\gamma = \beta^{-1}$ known a priori to be contained in an interval $[\gamma_{low}, \gamma_{high}]$.   If $\beta \geq 1.54059$, or $\gamma_{high} \leq .6491$, then we can recover $\gamma = \beta^{-1}$ from either a bitstream corresponding to 0, or a pair of bitstreams $(x,-x)$, using Theorem $\eqref{mybigtheorem}$.  Of course we should not consider only the scheme $\eqref{iter}$, but rather a revised scheme which accounts for integrator leak on the (single) integrator used in the beta-encoder implementation (see Figure 1).    The revised beta-encoding scheme, with slightly different initial conditions, becomes
\begin{eqnarray}
u_1 &=& x \nonumber \\
b_1 &=& Q(u_1) \nonumber \\
\textrm{for $j \geq 1$ }: u_{j+1} &=& \lambda \beta(u_j - b_j) \nonumber \\
b_{j+1} &=& Q(u_{j+1}) 
\label{iterlambda}
\end{eqnarray}
where $\lambda$ is an unknown parameter in $[.9,1]$.  As long as $\tilde{\beta} = \lambda \beta > 1$, we still have that $|x - \sum_{i=0}^{N} b_{i+1} \gamma^{i}| \leq C_{\tilde{\beta}} \tilde{\beta}^{-N}$ where $C_{\tilde{\beta}} = \frac{1}{\tilde{\beta}-1}$; furthermore, we can use Theorem $\eqref{mybigtheorem}$ to recover $\tilde{\beta} = \lambda \beta$ in $\eqref{iterlambda}$, although the specific values of $\lambda$ and $\beta$ cannot be distinguished, just as the specific values of $\lambda_1$ and $\lambda_2$ in the expression for $\gamma$ could not be distinguished in GRE.

\subsection{Remarks}
Figure 2 suggests that the first positive roots of the $P_n(t)$ serve as exponentially precise approximations to $\gamma$ for values of $\gamma$ greater than $.6491$; Figure 4 suggests that the first positive root of $P_n(t)$ will approximate values of $\gamma$ up to $\gamma = .75$.  Furthermore, these figures suggest that the constants $C_1$ and $C_2$ of $\eqref{search}$ in the exponential convergence of these roots to $\gamma$ can be made much sharper, even for larger values of $\gamma$. This should not be surprising, considering that nowhere in the analysis of the previous section did we exploit the specific structure of beta-expansions obtained via the particular recursions $\eqref{iterlambda}$ and $\eqref{(simple recurs)}$, such as the fact that such sequences $(b_n)$ cannot contain infinite strings of successive 1's or -1's.  It is precisely power series with such infinite strings that are the "extremal cases" which force the bound of $.6491$ in Theorem $\eqref{trans}$.  It is difficult to provide more refined estimates for the constants $C_1$ and $C_2$ of $\eqref{search}$ in general, but in the idealized setting where beta-expansions of $0$ are available via the ideal GRE scheme $\eqref{(simple recurs)}$ without leakage,  or via the beta-encoding $\eqref{(simple recurs)}$ with $\beta = \phi$, the beta-expansions of $0$ have a very special structure:
\begin{theorem}
Consider the ideal GRE recursion $\eqref{(simple recurs)}$ with input $u_0 = u_1 = 0$ and $Q(u,v) = Q_{\alpha}^{\nu}(u,v)$, or the beta-encoder recursion $\eqref{iter}$ with $\beta = \phi$, $u_0 = 0$, and $Q(u) = Q^{\nu}(u)$.   As long as $\alpha >  \nu$ in $\eqref{(simple recurs)}$, or $\nu \leq 1$ in $\eqref{iter}$, then for each  $j \in \{0,1, ... \}$, $b_{3j}$ is equal to $-1$ or $+1$, and $b_{3j + 1} = b_{3j+2} = -b_{3j} $.    
\label{period}
\end{theorem}
The proof is straightforward, and we omit the details.
\\ \\
Theorem $\eqref{period}$ can be used to prove directly that $\gamma$ must be the first positive root of the polynomials $P_N(t) = b_1 + \sum_{j=1}^{3N} b_{j+1} t^j$, when $\gamma = \phi^{-1}$ in either $\eqref{(simple recurs)}$ or $\eqref{iter}$.  Indeed, $P_N(t)$ can be factored as follows:
\begin{eqnarray}
P_N(t)  &=& \sum_{j=0}^{N} (b_{3j+1} - b_{3j+1} t - b_{3j+1} t^2) t^{3j} \nonumber \\
        &=& (1 - t - t^2) \sum_{j=1}^{N} b_{3j+1} t^{3j} \nonumber \\
        &=& (1 - t - t^2) R_N(t)
\end{eqnarray}
where $R_N(t)$ is a polynomial with random coefficients of the form $\sum_{j=0}^{N} \pm t^{3j}$.
\\
\\
$P_N(t) = (1 - t - t^2)R_N(t)$ clearly has a root at $t = \phi^{-1}$, and this root must be the only root of $P_N(t)$ on $[0, \phi^{-1})$, since on this interval $R_N(t)$ is bounded away from $0$ by $|R_N(t)| \geq 1 - \frac{t^3}{1 - t^3} \geq 1 - \frac{\phi^{-3}}{1 - \phi^{-3}} \approx .691$.  We can also obtain a lower bound on $|P_N'(\phi^{-1})| = -(1 + 2\phi^{-1})R_N(\phi^{-1})$ by $|P_N'(\phi^{-1})| \geq |1 + 2\phi^{-1}| |1 -  \frac{\phi^{-3}}{1 - \phi^{-3}}| \geq 1.545$.  Note that this bound holds also for the infinite sum $F(t) = b_1 + \sum_{j=1}^{\infty} b_{j+1} t^j$, and that the bound of $|F'(\gamma)| \geq 1.545$ is much sharper than the bound on $|F'(\gamma)|$ given by Theorem $\eqref{trans}$; e.g., $\delta_{.63} = .07$, and $\delta_{.6491} = .00008$ (see $\cite{1}$).  Similar bounds on the derivatives $|P_N'(\gamma)|$ and $|F_N'(\gamma)|$ corresponding to beta-expansions of $0$ in a base $\gamma$ close to $\phi^{-1}$ should hold, leading to sharper estimates on the constants $C_1$ and $C_2$ of $\eqref{search}$ in the case where beta-expansions of $0$ are available.    

\section{Conclusion}
In this paper, we have shown that golden ratio encoders are robust with respect to leakage in the delay elements of their circuit implementation.  Although such leakage may change the base $\gamma$ in the reconstruction formula $|y - \sum_{j=1}^{N} b_j \gamma^j| \sim  O(\gamma^N)$ , we have shown that exponentially precise approximations $\tilde{\gamma}$ to $\gamma$ can be obtained from the bitstreams of a pair $(x,-x)$, and such approximations $\tilde{\gamma}$ are sufficient to reconstruct subsequent input $y$ by $|y - \sum_{j=1}^{N} b_j \tilde{\gamma}^j| \sim  O(\gamma^N)$.   
\\
\\
Our method can be extended to recover the base $\beta$ in general beta-encoders, as long as $\beta$ is known a priori to be sufficiently large; e.g. $\beta \geq 1.54$.   Our method is similar to the method proposed in $\cite{3}$ for recovering $\beta$ in beta-encoders when $\{0,1\}$-quantizers are used, except that our method does not require a fixed reference level, which is difficult to measure with high precision in practice.

\begin{figure*}[htbp]
    \mbox{
      \subfigure[]{ \includegraphics[width=2.5in]{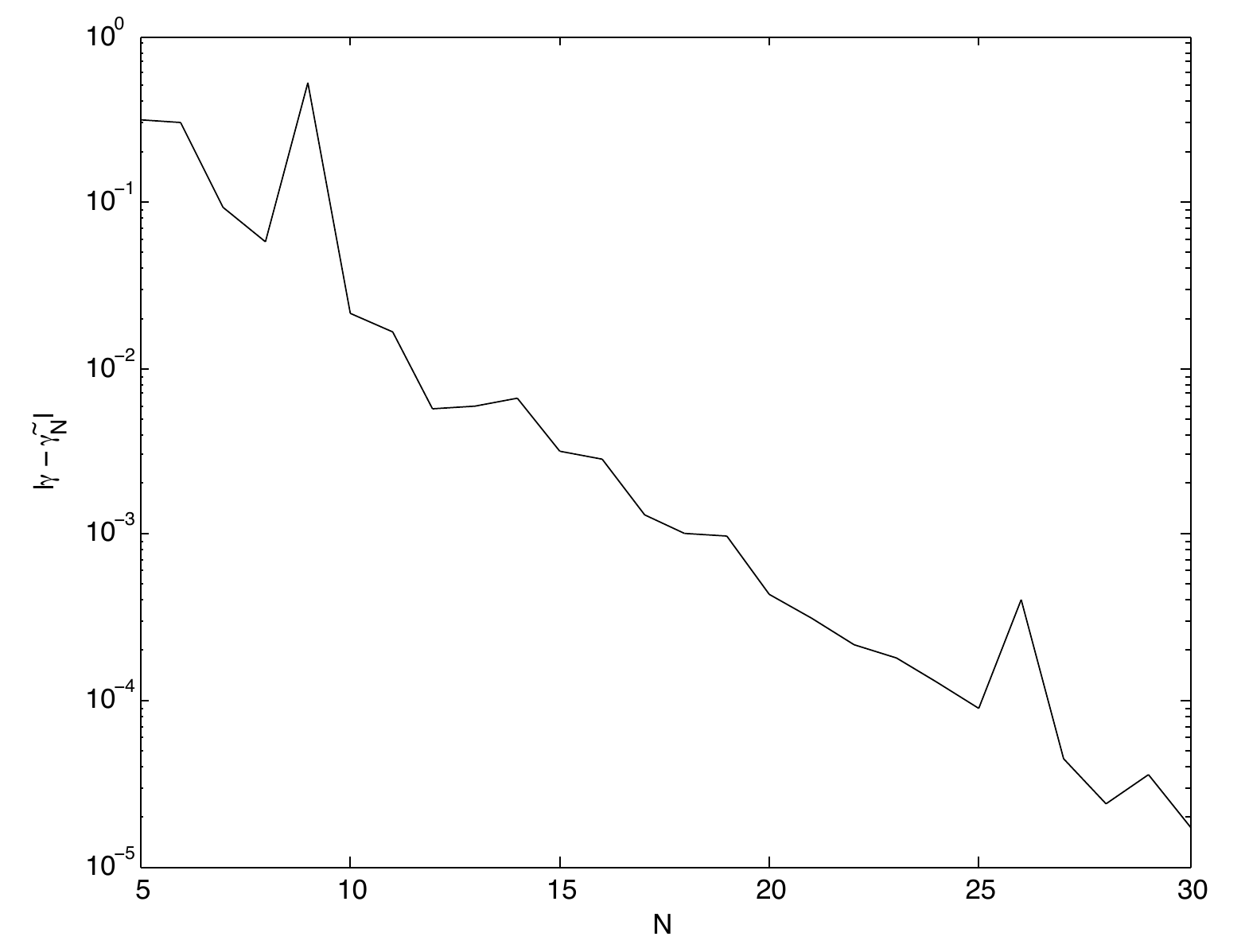}}
      \subfigure[]{\includegraphics[width=2.5in]{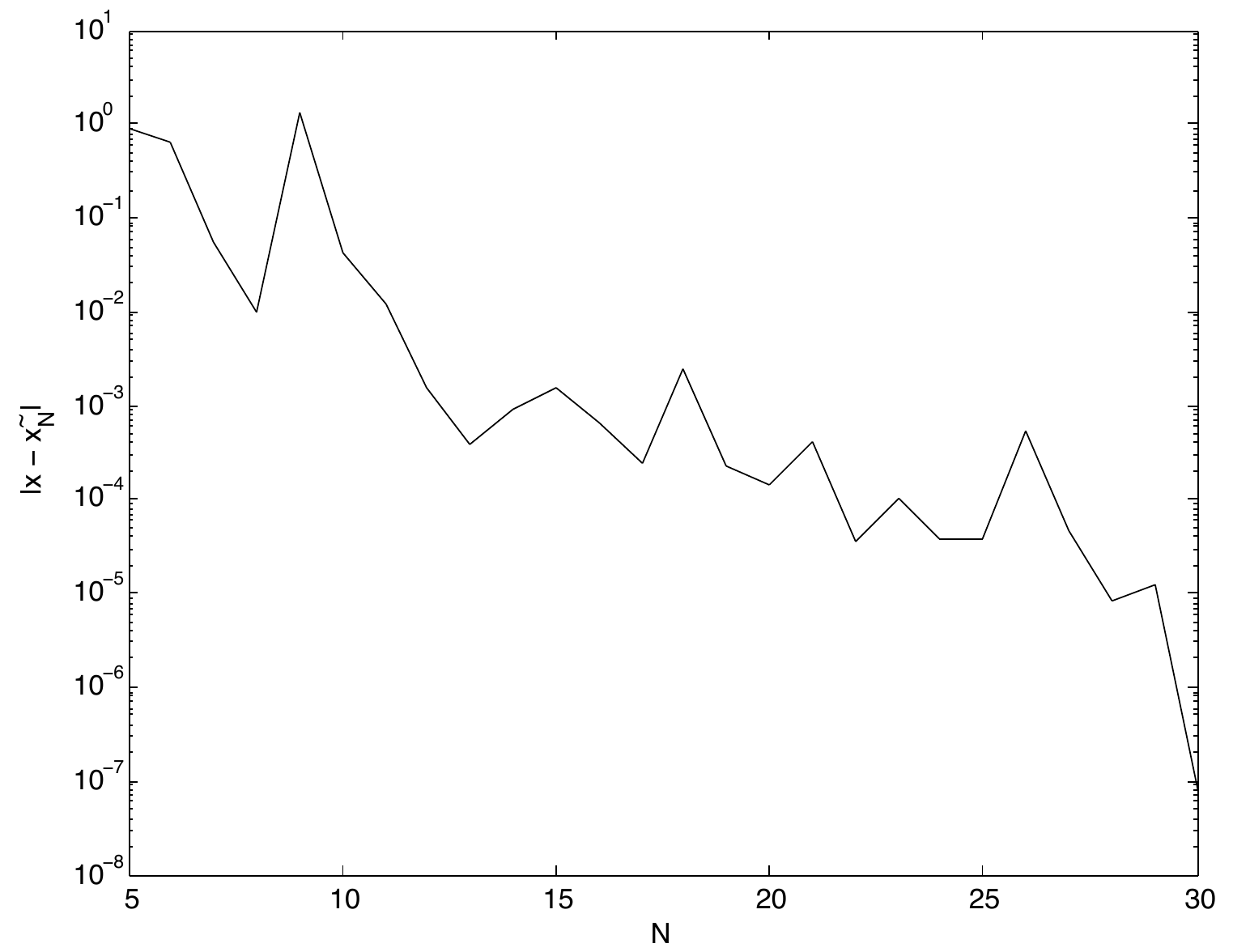}} 
      }
    \caption{ {\small (a) $N$ versus $ | \gamma - \tilde{\gamma}_N | $, where $\tilde{\gamma}_N$ is obtained via a 10 step Newton Method approximation to the first root of the polynomial $\bar{P}_N$ starting from $x_0 = .618$.  For each $N$, $| \gamma - \tilde{\gamma}_N |$ is the worst-case error among 100 experiments corresponding to different $(x_j,-x_j)$ pairs and different $\gamma_j$ chosen randomly from $[.618,.7]$.  (b) $N$ versus $ | x - \tilde{x}_N | $, where $\tilde{x}_N$ is reconstructed from $\tilde{\gamma}_N$ using the first $N$ bits.  The quantizer used in this experiment is $Q_{\alpha}^{\nu}(u,v)$ with $\nu = .3$ and $\alpha$ chosen randomly in the interval $[1.7,2]$.} }
 \label{figure4}
\end{figure*}

\begin{figure*}[htbp]
\mbox{
\subfigure[]{\includegraphics[width=2in]{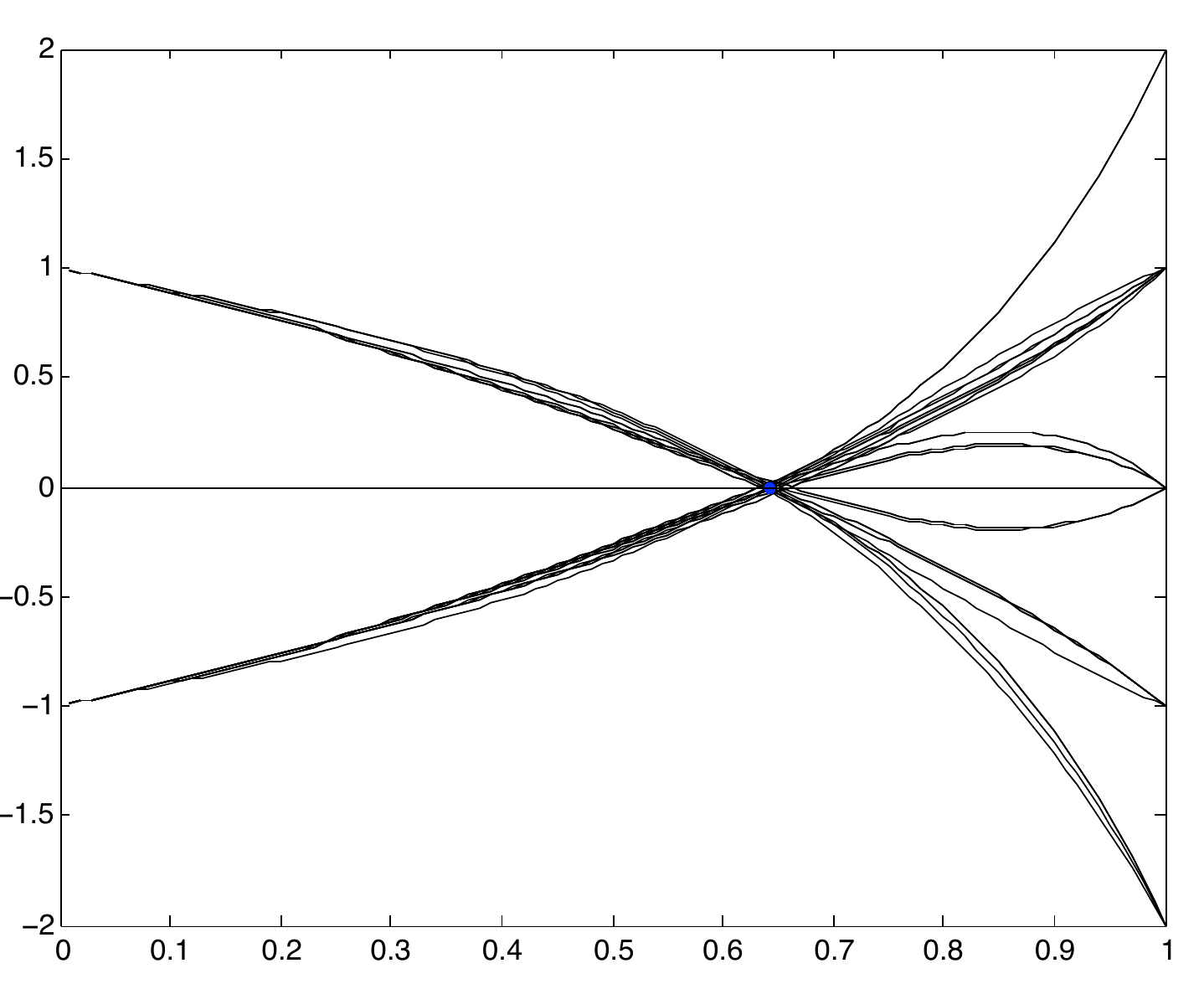}}
\subfigure[]{\includegraphics[width=2in]{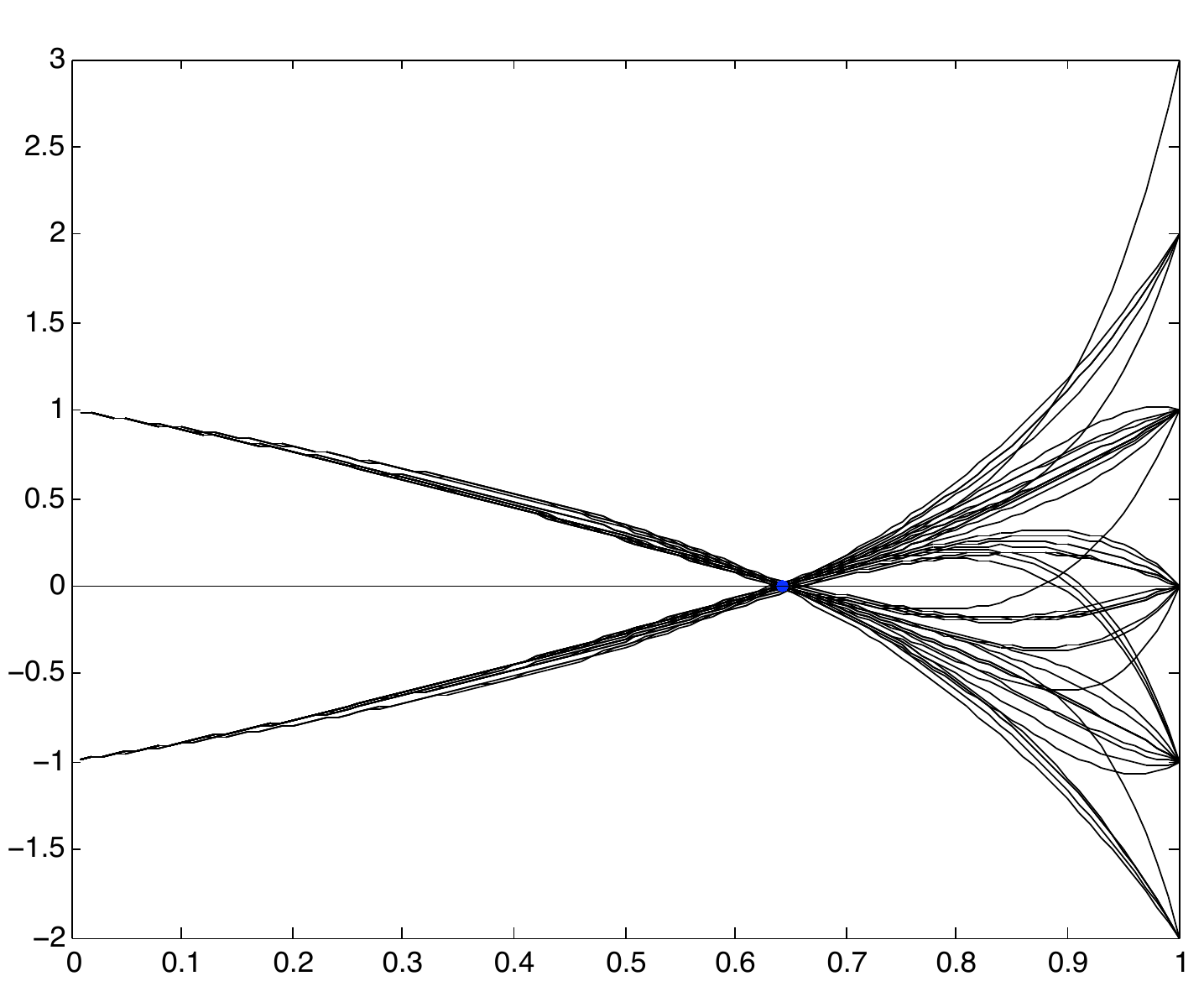}} 
\subfigure[]{\includegraphics[width=2in]{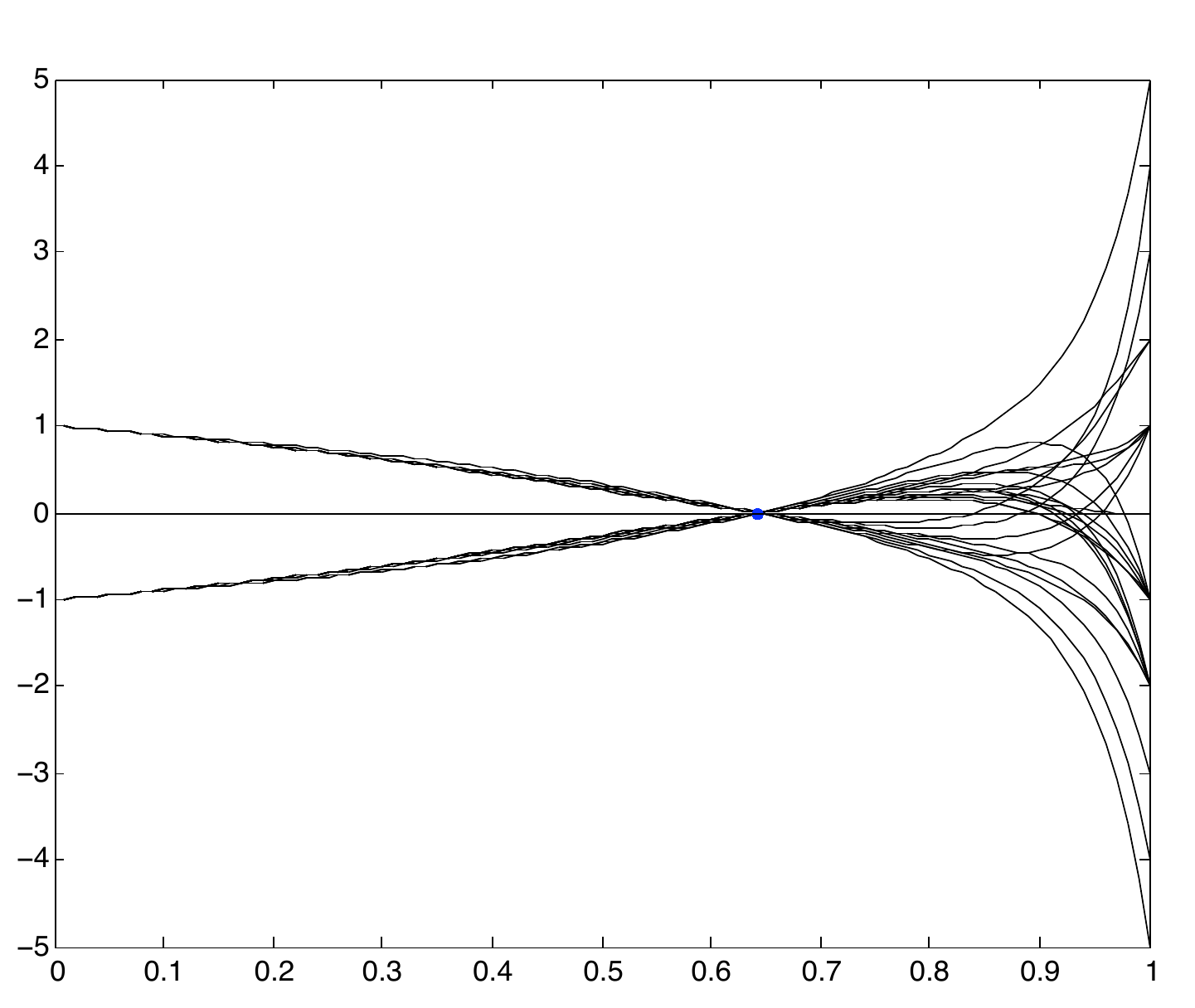}}
}
\caption{{\small The graphs of (a) $\bar{P}_{8}$, (b) $\bar{P}_{16}$, and (c) $\bar{P}_{32}$ for several pairs $(x_j,-x_j)$ corresponding to $\gamma = .64575$.  These polynomials can have several roots on the unit interval, but the first of these roots must become exponentially close to $\gamma$ as $N$ increases.    The quantizer used is $Q_{\alpha}^{\nu}(u,v)$ with parameter values $\nu = .3$, and $\alpha = 2$.}}
\label{figure 5}
\end{figure*}

\begin{figure*}[htbp]
\mbox{
\subfigure[]{\includegraphics[width=2in]{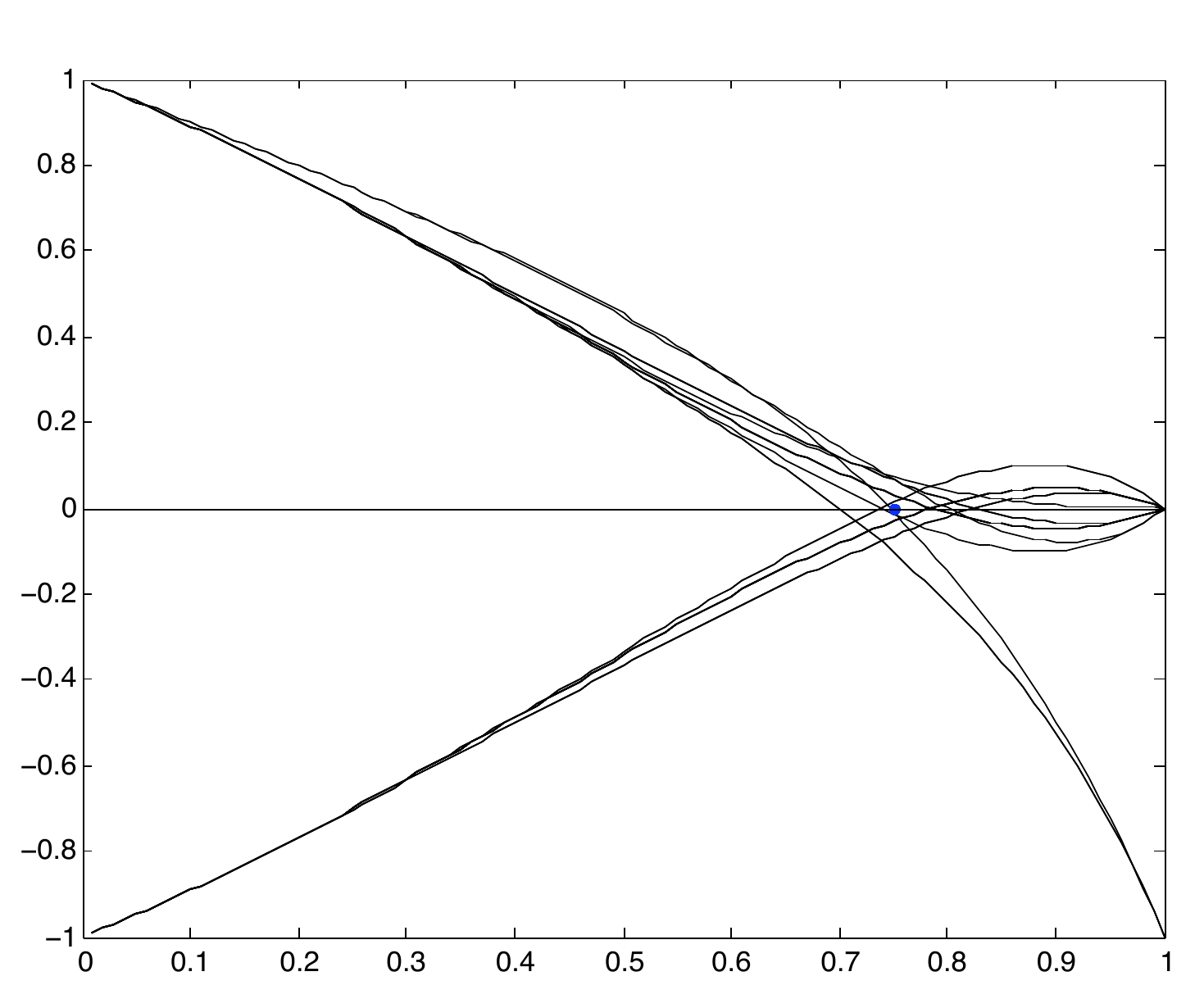}}
\subfigure[]{\includegraphics[width=2in]{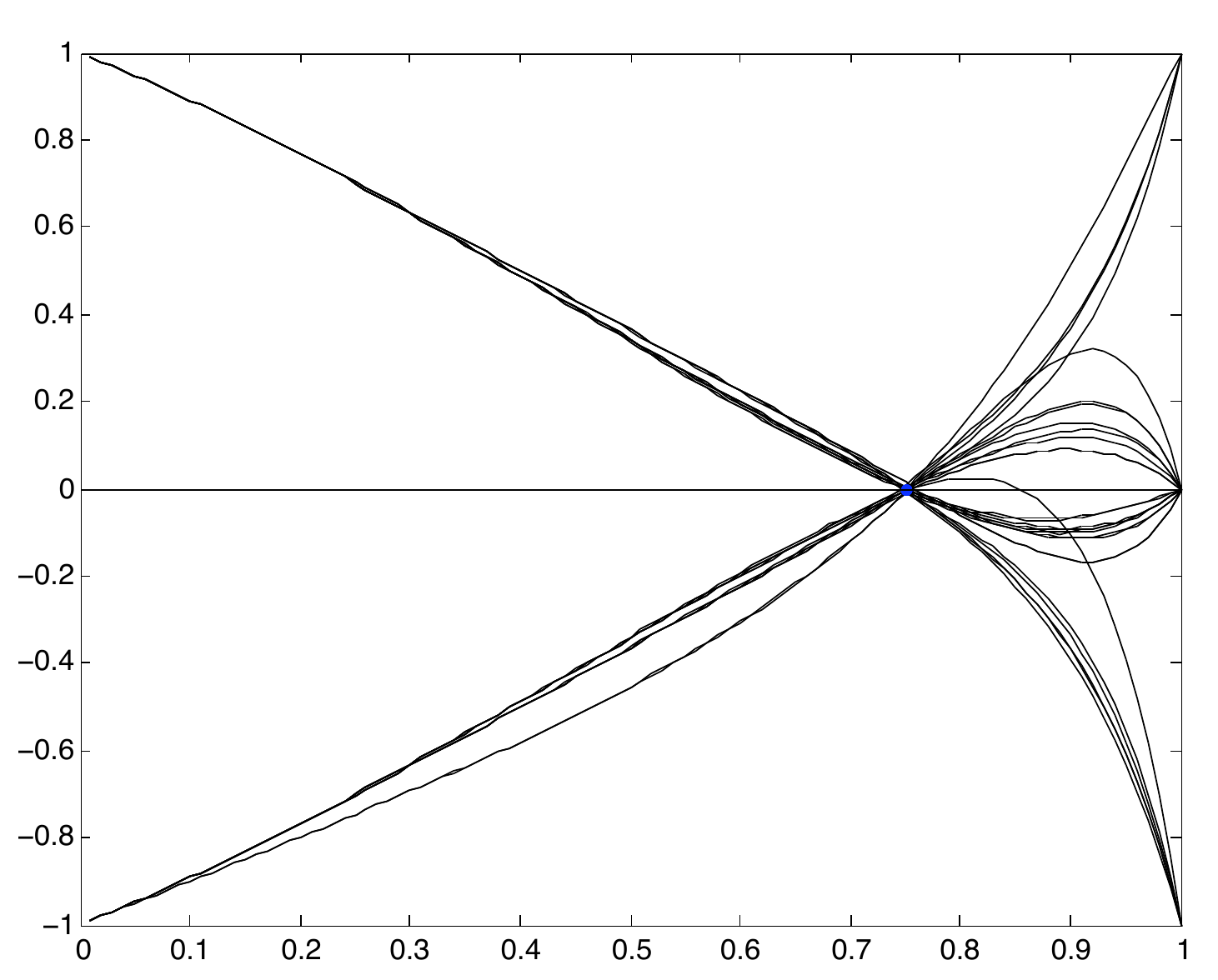}} 
\subfigure[]{\includegraphics[width=2in]{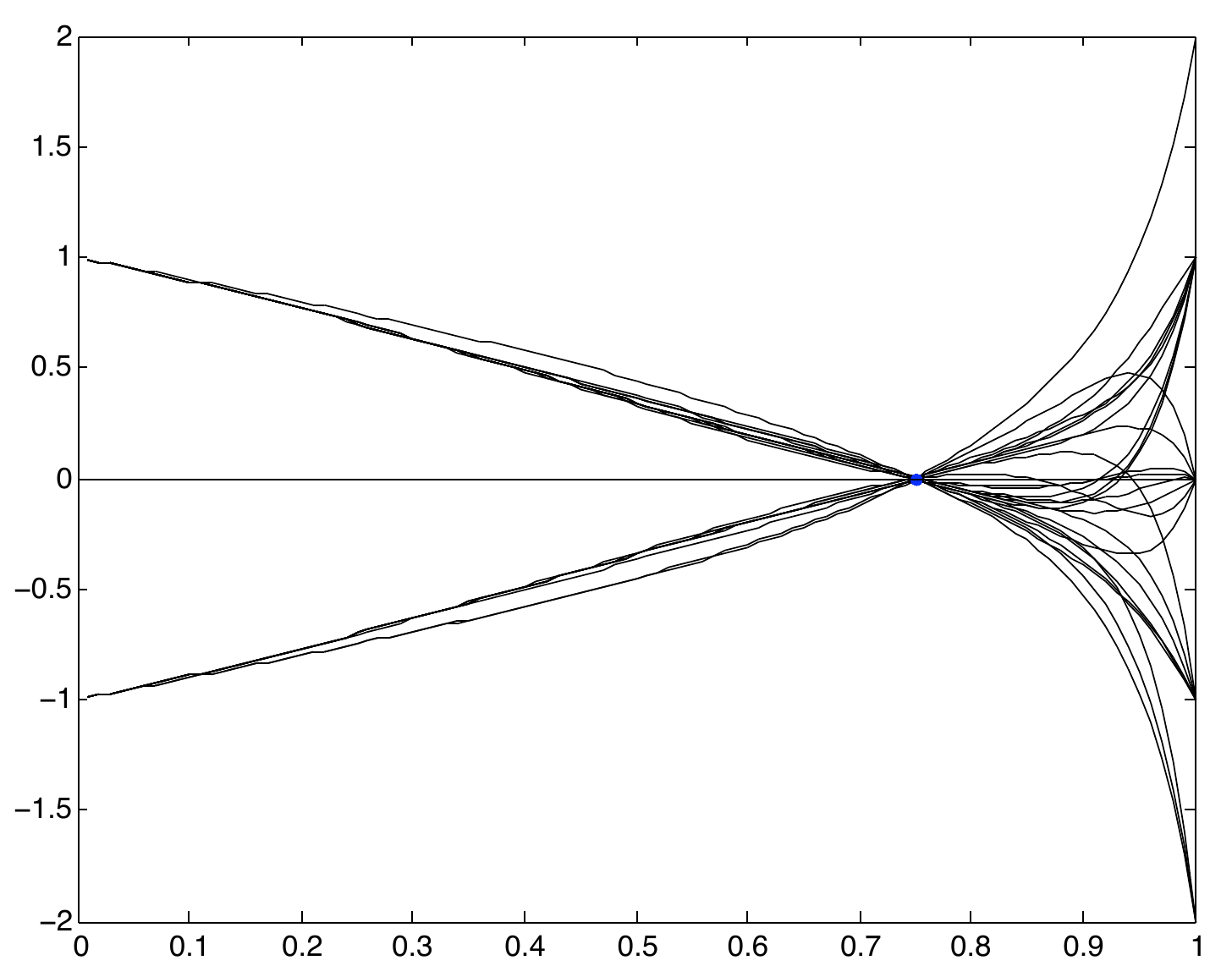}}
}
\caption{{\small The graphs of (a) $\bar{P}_{8}$, (b) $\bar{P}_{16}$, and (c) $\bar{P}_{32}$ for several pairs $(x_j,-x_j)$ corresponding to $\gamma = .75$.  These graphs suggest that the first positive root of these polynomials becomes exponentially close to $\gamma$ as $N$ increases, although we do not have a proof of this result for values of $\gamma$ greater than approximately $.65$.   The quantizer used is $Q_{\alpha}^{\nu}(u,v)$ with parameter values $\nu = .3$, and $\alpha = 2$.}}
\label{figure 6}
\end{figure*}

\section{{\bf Appendix}: Proof of Theorem $\eqref{range}$}

\begin{figure*}[htbp]
\begin{center}
\includegraphics[width=4in]{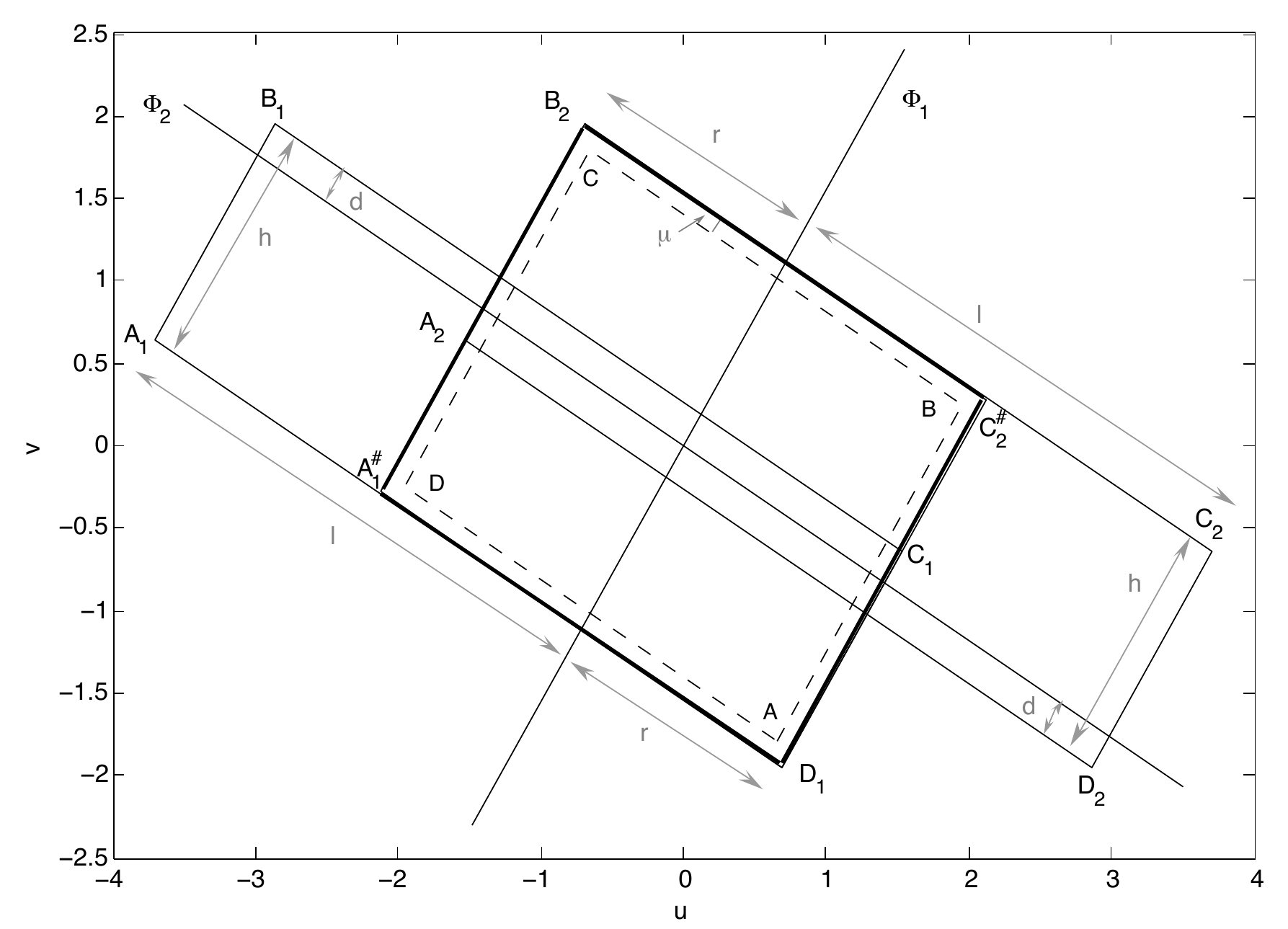}
\caption{ {\small The rectangle $R = A_1^{\#} B_2 C_2^{\#} D_1$ (bold outline) is a positively invariant set for the map $T_{\alpha}^{\nu}$.  In fact, $T_{\alpha}^{\nu} (R) $ is contained in the smaller rectangle $ABCD$ (dashed outline), illustrating that the revised GRE scheme $\eqref{(lambda recurs)}$ is robust with respect to small additive errors.  Here, $\Phi_1$ and $\Phi_2$ are the eigenvectors of the matrix in A in $\eqref{matrix}$ below.  In this figure, $\lambda_1 = \lambda_2 = .96$, and $\mu$ is taken to be $.0625$. The length parameters $l, r, h,$ and $d$ will be discussed later.} }
\label{Figure 5}
\end{center}
\end{figure*}

In this section we prove Theorem $\eqref{range}$, which provides a range within which the "flakiness" parameter $\nu$ and the "amplifier" parameter $\alpha$ in the quantizer $Q_{\alpha}^{\nu}(u,v)$ can vary from iteration to iteration, without changing the fact that the sequences $(u_n)$ produced by the scheme $\eqref{(lambda recurs)}$ will be bounded.  The derived range for $\nu$ and $\alpha$ is independent of the specific values of the parameters $(\lambda_1, \lambda_2) \in [.9,1]^2$ in $\eqref{(lambda recurs)}$.  
\\
\\
The techniques of this section are borrowed in large part from those used in $\cite{5}$ to prove a similar result for the ideal GRE scheme $\eqref{(simple recurs)}$; i.e., taking $\lambda_1 = \lambda_2 = 1$ in $\eqref{(T_Q)}$.  As is done in $\cite{5}$, we first observe that the recursion formula $\eqref{(lambda recurs)}$ corresponding to the leaky GRE  is equivalent to the following piecewise affine discrete dynamical system on $\mathbb{R}^2$ :

\begin{equation}
		\left[\begin{array}{cl}
	u_{n+1} \\
	u_{n+2}
	   \end{array}\right]
           =T_{\alpha}^{\nu}\left[\begin{array}{cl}
	u_n \\
	u_{n+1}
	   \end{array}\right]
           \label{(DS)}
\end{equation}

where

\begin{equation}
	T_{\alpha}^{\nu}:	\left[\begin{array}{cl}
	u \\
	v
	   \end{array}\right]
           \rightarrow       \left[\begin{array}{cc}
	0 & 1 \\
	\lambda_1 \lambda_2 & \lambda_1
	   \end{array}\right] 
           \left[\begin{array}{cl}
	u \\
	v
	   \end{array}\right] 
           - Q_{\alpha}^{\nu}(u,v) \left[\begin{array}{cl}
	0 \\
	1
	   \end{array}\right]
             \label{(T_Q)}
 \end{equation}
We will construct a class of subsets $R = R(\mu) = R(\lambda_1,\lambda_2,\mu)$ of $\mathbb{R}^2$ for which $T_{\alpha}^{\nu}(R) + B_{\mu}(0) \subset R$, where $B_{\mu}(0)$ is the disk of radius $\mu$ centered at the origin; i.e. $B_{\mu}(0) = \{ (u,v): u^2 + v^2 \leq \mu^2 \}$.  Note that these sets are not only { \it positively invariant sets } of the map $T_{\alpha}^{\nu}$, but have the additional property that if $(u_0, v_0) \in R(\mu)$, the image $(u_{n+1}, v_{n+1}) = T_{\alpha}^{\nu} (u_n, v_n)$ may be perturbed at any time within a radius of $\mu$ (for example, by additive noise), and the the resulting sequence $(u_n)_{n=0}$ will still remain bounded within $R(\mu)$ for all time $n$.  
\\
\\
We refer the reader to Figure 5 as we detail the construction of the sets $R(\mu)$.   Rectangles $A_1 B_1 C_1 D_1$ and $A_2 B_2 C_2 D_2$ in Figure 5 are designed so that their respective images under the affine maps $T_1$ and $T_2$ (defined below) are both equal to the dashed rectangle $ABCD$.  
\begin{eqnarray}
	T_1: \left[\begin{array}{cl}
	u \\
	v
	   \end{array}\right]   \rightarrow   \left[\begin{array}{cc}
	0 & 1 \\
	\lambda_1 \lambda_2 & \lambda_1
	   \end{array}\right] 
           \left[\begin{array}{cl}
	u \\
	v
	   \end{array}\right] 
           - \left[\begin{array}{cl}
	0 \\
	1
	   \end{array}\right],  \nonumber \\
          T_2 : \left[\begin{array}{cl}
	u \\
	v
	   \end{array}\right]   \rightarrow   \left[\begin{array}{cc}
	0 & 1 \\
	\lambda_1 \lambda_2 & \lambda_1
	   \end{array}\right] 
           \left[\begin{array}{cl}
	u \\
	v
	   \end{array}\right] 
           + \left[\begin{array}{cl}
	0 \\
	1
	   \end{array}\right].
	   \label{T1}
	   \end{eqnarray}     
That is, rectangles $A_1 B_1 C_1 D_1$ and $A_2 B_2 C_2 D_2$ will satsify $T_1(A_1 B_1 C_1 D_1) = ABCD = T_2(A_2 B_2 C_2 D_2)$. More specifically, $T_1(A_1) = T_2(A_2) = A$, $T_1(B_1) = T_2(B_2) = B$, and so on.  Since the rectangle $R = A_1^{\#} B_2 C_2^{\#} D_1$ is contained within the union of $A_1 B_1 C_1 D_1$ and $A_2 B_2 C_2 D_2$, $R$ is a positively invariant set for any map $T(u,v)$ satisfying
\begin{equation}
T(u,v) =  \left\{\begin{array}{cl}
	T_1(u,v),  & (u,v)  \in R \setminus A_2 B_2 C_2 D_2   \\
	T_2(u,v),   &(u,v) \in R  \setminus A_1 B_1 C_1 D_1 \\
	T_1(u,v) \textrm{  } or \textrm{  } T_2(u,v), &(u,v) \in H \\
	   \end{array}\right.
	   \label{admin}
\end{equation}  
where $H =  A_1 B_1 C_1 D_1 \cap A_2 B_2 C_2 D_2$.  In particular, if the parameters $\alpha$ and $\nu$ in the map $T_{\alpha}^{\nu}$ are chosen such that the intersection of $R$ and the strip $F = \{(u,v): -\nu < u + \alpha v < \nu \}$ is a subset of $H$, then $T_{\alpha}^{\nu}$ is of the form $\eqref{admin}$.  Indeed, $F$ is the region of the plane in which the quantizer $Q_{\alpha}^{\nu}(u,v)$ operates in flaky mode. This geometric setup is clarified with a figure, which we provide in Figure 6.

\begin{figure}[htbp]
\begin{center}
\includegraphics[width=3in]{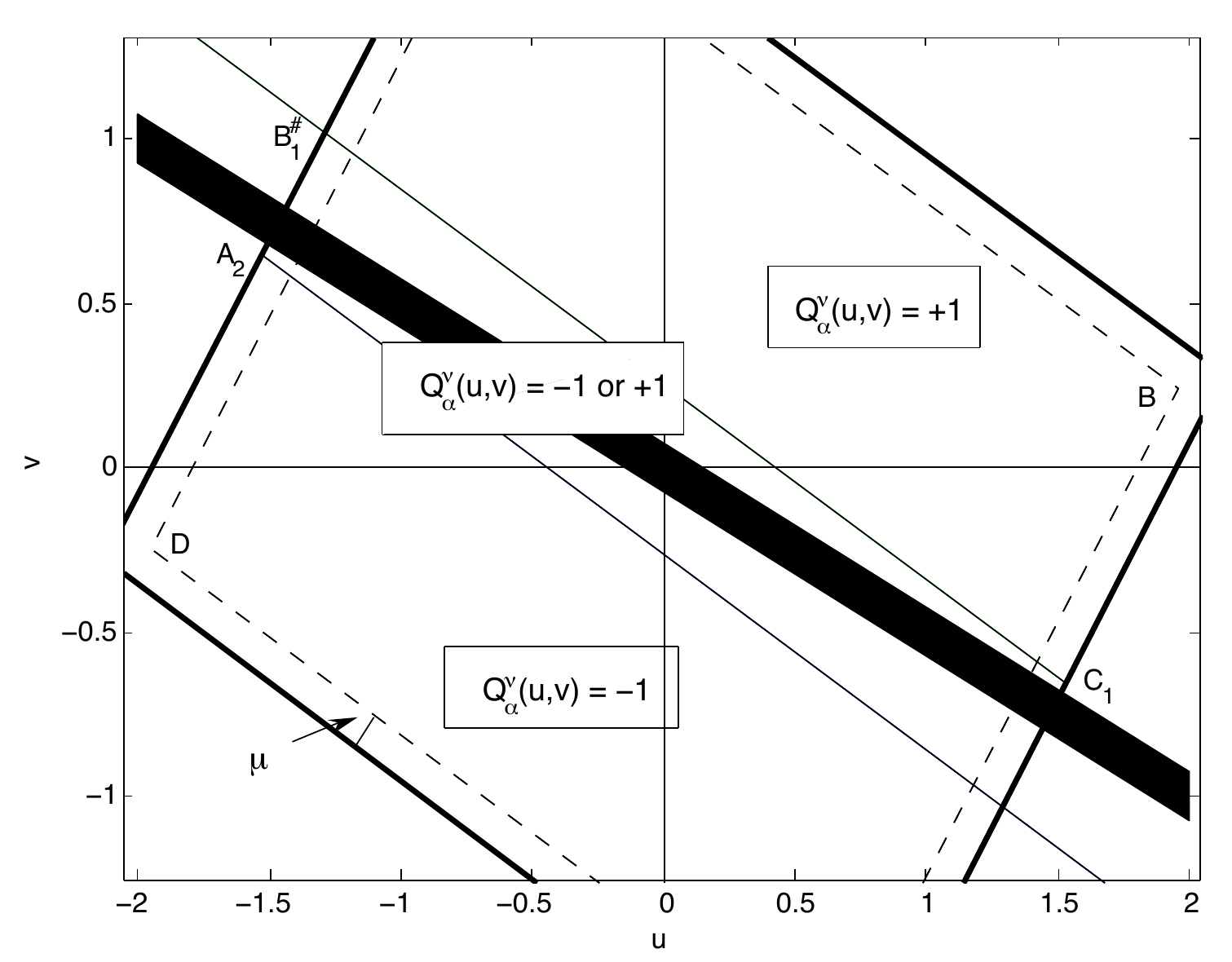}
\end{center}
\label{figure3}
\caption{ {\small The flaky quantizer $Q_{\alpha}^{\nu}(u,v)$ outputs either $-1$ or $+1$ when the input $(u,v)$ belongs to the strip $F = \{(u,v): -\nu < u + \alpha v < \nu \}$ (shaded above).  Here $\mu = .0625$, $\alpha = 2$, and $\nu = .15$.  For these parameter values, the intersection of $F$ and $R$ is a subset of  $A_1 B_1 C_1 D_1 \cap A_2 B_2 C_2 D_2$, meaning that the scheme $\eqref{(lambda recurs)}$ produces bounded sequences $(u_n)$, when implemented with this quantizer, and with input $(u_0,u_1) \in R$.}}
\end{figure}

It remains to verify the existence of at least one solution to the setup in Figure 5 for each $(\lambda_1,\lambda_2) \in [.9,1]^2$.  This can be done, following the lead of $\cite{5}$, in terms of the parameters defined in Figure 5.
\\
\\
Note that the matrix 
\begin{equation}
A =  \left[\begin{array}{cc}
	0 & 1 \\
	\lambda_1 \lambda_2 & \lambda_1
	   \end{array}\right]
	   \label{matrix}
\end{equation}
in $\eqref{T1}$ has as eigenvalues $\epsilon_1 = \frac{\lambda_1 + \sqrt{\lambda_1^2+ 4 \lambda_1\lambda_2} }{2}$ and $- \epsilon_2  = -\big( \frac{- \lambda_1 + \sqrt{\lambda_1^2 - 4\lambda_1\lambda_2}}{2} \big)$.   In particular, when $\lambda_1 = \lambda_2 = 1$, we have that $\epsilon_1 = \phi \approx 1.618$ and $\epsilon_2 = \phi^{-1} \approx .618$.  The eigenvalues $\epsilon_1$ and $-\epsilon_2$ have respective normalized eigenvectors 
\begin{center}
$\Phi_1 = s_1^{-1}\left[\begin{array}{cl}
	1 \\
	\epsilon_1
	   \end{array}\right]$ 
	   and $\Phi_2 = - s_2^{-1} \left[\begin{array}{cl}
	1 \\
	\epsilon_2
	   \end{array}\right]$,
\end{center}
where $s_1 = \sqrt{1 + \epsilon_1^2}$ and $s_2 = \sqrt{1 + \epsilon_2^2}$.   It follows that the affine map $T_1$ acts as an expansion by a factor of $\epsilon_1$ along $\Phi_1$ and a reflection followed by contraction by a factor of $\epsilon_2$ along $\Phi_2$, followed by a vertical translation of +1.  $T_2$ is the same as $T_1$ except with a vertical translation of -1 instead of +1.  After some straightforward algebraic calculations, the mapping relations described above imply that the parameters in Figure 5 are given by
\begin{eqnarray}
h  &=& \frac{2\mu}{1-\epsilon_1} + \frac{2s_1}{\epsilon_1(\epsilon_1-1)(\epsilon_1 + \epsilon_2)} \nonumber \\
d  &=& \frac{\mu}{1-\epsilon_1} + \frac{s_1(2 - \epsilon_1)}{\epsilon_1(\epsilon_1-1)(\epsilon_1 + \epsilon_2)} \nonumber \\
l  &=& \frac{\mu}{1 - \epsilon_2} + \frac{s_2(2 - \epsilon_2)}{\epsilon_2(1 - \epsilon_2)(\epsilon_1 + \epsilon_2)} \nonumber \\
r &=& \frac{\mu}{1 - \epsilon_2} + \frac{s_2}{(1-\epsilon_2)(\epsilon_1 + \epsilon_2)}
\label{points}
\end{eqnarray}
The existence of the overlapping region $H = A_1 B_1 C_1 D_1 \cap A_2 B_2 C_2 D_2$ is equivalent to the condition $d > 0$ which in turn is equivalent to the condition $\mu < \frac{s_1(2-\epsilon_1)}{\epsilon_1(\epsilon_1 + \epsilon_2)}$.  This expression is minimized over the range $(\lambda_1,\lambda_2) \in V = [.9,1]^2 $ when $\lambda_1 = \lambda_2 = 1$, in which case this constraint simplifies to $\mu < \frac{2 - \phi}{\sqrt{\phi^2+1}} \approx .2008$. 
\\
\\
We are interested primarily in quantizing numbers in $[-1,1]$, so that this is the range of interest for the input $u_1 = x$; for $u_0$ we simply take $0$.   The set $\{0\} \times [-1,1] \subset  \bigcap_{(\lambda_1,\lambda_2) \in V} R(\lambda_1,\lambda_2,\mu)$ is equivalent to the condition that the line  $v = mu + b$ which passes through the points $B$ and $C$ in Figure 5 have y-intercept $b \geq 1$ for all $(\lambda_1,\lambda_2) \in V$.   This line has slope $m = -\epsilon_2$, and passes through the point $C$, so that its y-intercept is given by $b = s_1^{-1}(\epsilon_1 + \epsilon_2)(h - d)$.  Rearranging terms, this implies that $b \geq 1$ if and only if 
\begin{equation}
\mu \leq \frac{(\sqrt{1 + \epsilon_1^2})(2 - \epsilon_1)}{\epsilon_1 + \epsilon_2}.
\label{line}
\end{equation}
The right hand side of the above inequality is bounded below over the range $(\lambda_1,\lambda_2) \in V$ by $\frac{\sqrt{1 + (.9\phi)^2}(2 - \phi)}{\sqrt{5}} \approx .301$, so that indeed $\{0\} \times [-1,1] \subset R(\lambda_1,\lambda_2,\mu)$ is satisfied independent of $(\lambda_1,\lambda_2) \in V$, for all admissable $\mu \in [0,.2008] $.   
\\
\\
In practice, the "flaky" parameter $\nu$ of the quantizer $Q_{\alpha}^{\nu}$ is not known exactly; we will however be given a tolerance $\delta$ for which it is known that $|\nu| \leq \delta$.  It follows that for each $\delta$, we would like a range $(\alpha_{min}, \alpha_{max})$ for the "amplifier" $\alpha$ such that the GRE scheme $\eqref{(lambda recurs)}$, implemented with quantizer $Q_{\alpha}^{\nu}$, produces bounded sequences $(u_n)$ for all values of $(\lambda_1,\lambda_2) \in V$.  For a particular choice of $(\lambda_1, \lambda_2)$, it is not hard to derive an admissable range for $\alpha$ in terms of the eigenvalues $\epsilon_1$ and $\epsilon_2$ for fixed $\delta$.  We will take $\mu = 0$ in the following analysis for the sake of simplicity.  First of all, the tolerance $\delta$ must be admissable, i.e. the coordinate $(\delta, 0)$ should lie within the region $H$.  Indeed, for $\nu \in [0,\delta]$, $\alpha$ can vary within a small neighborhood of $\frac{1}{\epsilon_2}$, as long as the line with slope $-\frac{1}{\alpha}$ which passes through $(0,\delta)$ remains bounded above in the region $H$ by the line passing through $B_1$ and $C_1$.  In other words, the admissable range for $\alpha$ is obtained from the constraint $m_1 \leq -\frac{1}{\alpha} \leq m_2$, where $m_1$ is the slope  of the line through the points $\{ ( \delta, 0),  C_1 \}$ and $m_2$ is the slope of the line through the points $\{ ( \delta, 0 ), B_1^{\#} \}$ in Figure 6.   Rewriting $C_1$ and $B_1^{\#}$ in terms of the eigenvalues $\epsilon_1$ and $\epsilon_2$ via the relations $\eqref{points}$, we have that

\begin{equation}
L(\epsilon_1,\epsilon_2) \leq \alpha \leq U(\epsilon_1,\epsilon_1 + \epsilon_2)
\label{bounds}
\end{equation}
where
\begin{eqnarray}
&&L(x,y) = N(x,y)/D(x,y) \nonumber \\
&=&  \frac{x(x - 1) - (2-x)(1 - y) + \delta x(x - 1)(x + y)(1 - y)}{x((2 - x)(1-y) + y(x-1))} \nonumber
\end{eqnarray}
and 
\begin{equation}
U(x,y) = \frac{2 + xy - 2y - \delta xy(x - 1)(1 - y + x)}{x(y - 2)}.
\end{equation}
 It is not hard to show that for any fixed $y \in [.9 \phi + .9 \phi^{-1}, \phi + \phi^{-1}] \approx [2.0124, 2.236]$, the function $U(x,y)$ increases as a function of $x$ over the range $x \in [.9\phi, \phi] \approx [1.456, 1.618]$, as long as $\nu \leq  .75$.  It follows that the minimum of $U(x,y)$ over the rectangle $S = [1.456, 1.618] \times [2.0124, 2.236]$ occurs along the edge corresponding to $x = 1.456$.  But $U(1.456,y)$ decreases as a function of $y$ over the interval $[2.0124, 2.236]$, so that the minimum of $U(x,y)$ over the entire rectangle $S$ occurs at $(x,y) = (1.456, 2.236)$, giving the following uniform lower bound on $\alpha_{max}$ for $(\lambda_1,\lambda_2) \in V$:
\begin{eqnarray}
\alpha_{max} &\geq& U(1.456, 2.236) \nonumber \\
                         &\approx& 2.281 - .952 \delta.
\label{max}
\end{eqnarray} 
We proceed in the same fashion to derive a uniform upper bound on $\alpha_{min}$ over $(\lambda_1,\lambda_2) \in V$, except that we analyze the numerator and the denominator in the expression for $L(x,y)$ separately.  The numerator $N(x,y)$ increases as a function of $x$ over the range $x \in [1.456, 1.618]$ for fixed $y \in [.9\phi^{-1}, \phi^{-1}] \approx [.5562,.618]$, so that $N(x,y)$ achieves its maximum over the rectangle $T = [1.456, 1.618]\times [.5562,.618]$ along the edge corresponding to $x = 1.618$.   $N(1.618,y)$ increases as a function of $y$ over the range $y \in [.5562,.618]$, so that $N(x,y)$ achieves its maximum over $T$ at $(x,y) = (1.618, .618)$.   A similar analysis shows that the denominator $D(x,y)$ is minimized over $T$ at $(x,y) = (1.456, .618)$.  It follows that for $(\lambda_1,\lambda_2) \in [.9,1]^2$, $\alpha_{min}$ is bounded above by
\begin{eqnarray}
\alpha_{min} &\leq& N(1.618, .618)/D(1.456, .618) \nonumber \\
                        &\approx& 1.198(1+\delta)
\label{m}
\end{eqnarray}
Finally, we note that $\delta$ is admissable for all $(\lambda_1, \lambda_2)$ if and only if $L(\epsilon_1, \epsilon_2) \leq U(\epsilon_1, \epsilon_1 + \epsilon_2)$ for this value of $\delta$ and for all coresponding $\epsilon_1$ and $\epsilon_2$.  Thus an  admissable range for $\delta$ is obtained by equating the lower bound of $2.281 - .952 \delta$ in $\eqref{max}$ and the upper bound of $1.198(1+\delta)$ in $\eqref{m}$; namely, $\delta \in [0, .5037]$.

\section*{Acknowledgment}
The author would like to thank her advisor, Ingrid Daubechies, for the encouragement and insightful discussions which motivated much of this work.  The author also is grateful to Sinan G{\"u}nt{\"u}rk for his helpful feedback, as well as NSF for providing her with a Graduate Research Fellowship.   She is also indebted to Ozgur Yilmaz, along with the anonymous referees, for helpful comments that significantly improved the quality of the paper.

%\begin{figure}[htbp]
%\includegraphics[width=1in]{raw.jpg}

%{ \bf Rachel Ward} is a third year PhD student in the Program in Applied and Computational Mathematics (PACM) at Princeton University, working under the guidance of Prof. Ingrid Daubechies.  She was born in College Station, TX, where she lived until 2001.  She received a bachelors degree with highest honors in mathematics from the University of Texas at Austin in 2005, before starting her graduate work at Princeton University.  
%\end{figure}

\end{document}